\documentclass{amsart}

\usepackage{amsthm}
\usepackage{url}
\usepackage{amsfonts}
\usepackage{amssymb}
\usepackage[normalem]{ulem}
\usepackage{tensor}
\usepackage{bbm}
\usepackage{mathrsfs}
\usepackage{mathtools}
\numberwithin{equation}{section}
\usepackage{color}
\usepackage{hyperref}
\usepackage{graphicx}
\usepackage{caption}
\usepackage{subcaption}
\usepackage{subcaption}
\newtheorem{theorem}{Theorem}
\newtheorem{lemma}[theorem]{Lemma}

\newtheorem{remark}[theorem]{Remark}

\newcommand{\nnnorm}[2]{{\left\vert\kern-0.25ex\left\vert\kern-0.25ex\left\vert 
#2 
    \right\vert\kern-0.25ex\right\vert\kern-0.25ex\right\vert_{#1}}}

\newcommand{\nnorm}[2]{{\left\vert\kern-0.25ex\left\vert\kern-0.25ex\left\vert 
#2 
    \right\vert\kern-0.25ex\right\vert\kern-0.25ex\right\vert^{2}_{#1}}}
    
\newcommand{\aabs}[2]{\ensuremath{|#2|_{#1}}}
\newcommand{\norm}[2]{\ensuremath{\|#2\|_{#1}}}

\newcommand{\abs}[1]{\ensuremath{|#1|}}

\newcommand{\innernosym}[2]{\ensuremath{\big\langle #1,#2 \big\rangle}}
\newcommand{\inner}[2]{\ensuremath{\langle #1,#2 \rangle}_{H}}

\newcommand{\duals}[4]{\ensuremath { \tensor[_{#3}]{\langle#1,#2\rangle 
}{_{#4}}  }}

\newcommand{\dual}[2]{\ensuremath { \tensor[_{V^*}]{ \langle#1,#2\rangle }{_V}  
  }}
\newcommand{\tdual}[2]{\ensuremath { \tensor[_{V}]{ \langle#1,#2\rangle
}{_{V^*}} }}

\newcommand{\bR}{{\mathbb R}}

\newcommand{\bP}{{\mathbb P}}

\newcommand{\dd}{{\mathrm{d}}}

\newcommand{\cY}{\mathcal{Y}}

\newcommand{\cX}{\mathcal{X}}

\newcommand{\cL}{\mathcal{L}}

\newcommand{\cT}{\mathcal{T}}

\newcommand{\AM}{\ensuremath A_{\max}}
\newcommand{\Am}{\ensuremath A_{\min}}

\newcommand{\sC}{\mathscr{C}}

\newcommand{\sL}{\mathscr{L}}

\newcommand{\sB}{\mathscr{B}}

\newcommand{\sV}{\mathscr{V}}

\newcommand{\sF}{\mathscr{F}}
\newcommand{\sW}{\mathscr{W}}

\title[Numerical solution of the heat equation]{Numerical solution of parabolic 
problems based on a weak space-time formulation}

\author[S.~Larsson]{Stig Larsson}

\address{
  Department of Mathematical Sciences,
  Chalmers University of Technology and University of Gothenburg, 
  SE--412 96 Gothenburg,
  Sweden}

\email{stig@chalmers.se}

\author[M.~Molteni]{Matteo Molteni}

\address{
  Department of Mathematical Sciences,
  Chalmers University of Technology and University of Gothenburg,
  SE--412 96 Gothenburg, 
  Sweden}

\email{molteni@chalmers.se}

\keywords{inf-sup, space-time, superconvergence, quasi-optimality, finite 
element, error estimate, Petrov--Galerkin}

\subjclass[2010]{65M15, 65M60}

\begin{document}

\begin{abstract}
% PURPOSE OF THE PAPER
We investigate a weak space-time formulation of the 
heat equation and its use for the construction of a numerical 
scheme.
% METHOD
The formulation is based on a known weak space-time formulation, with
the difference that a pointwise component of the solution, which in
other works is usually neglected, is now kept. We investigate the role
of such a component by first using it to obtain a pointwise bound on
the solution and then deploying it to construct a numerical scheme.
% RESULTS
The scheme obtained, besides being quasi-optimal in the $L^2$ sense, is also
pointwise superconvergent in the temporal nodes. We prove \emph{a
  priori} error estimates and we present numerical experiments to
empirically support our findings.
\end{abstract}

\date{\today}

\maketitle
  
%%%%%%%%%%%%%%%%%%%%%%%%%%%%%%%%%%%%%%%%%%%%%%%%%%%%%%%%%%%%%%%%%%%%%%%%%%%%%%%%
%%%%%%%%%%%%%%%%%%%%%%%%%%%%%%%%%%%%%%%%%%%%%%%%%%%%%%%%%%%%%%%%%%%%%%%%%%%%%%%%
%%%%%%%%%%%%%%%%%%%%%%%%%%%%%%%%%%%%%%%%%%%%%%%%%%%%%%%%%%%%%%%%%%%%%%%%%%%%%%%%
%%
%%
%% Introduction: A FEW WORDS ABOUT WHAT I AM TALKING ABOUT
%%
%%
%%%%%%%%%%%%%%%%%%%%%%%%%%%%%%%%%%%%%%%%%%%%%%%%%%%%%%%%%%%%%%%%%%%%%%%%%%%%%%%%
%%%%%%%%%%%%%%%%%%%%%%%%%%%%%%%%%%%%%%%%%%%%%%%%%%%%%%%%%%%%%%%%%%%%%%%%%%%%%%%%
%%%%%%%%%%%%%%%%%%%%%%%%%%%%%%%%%%%%%%%%%%%%%%%%%%%%%%%%%%%%%%%%%%%%%%%%%%%%%%%%
\section{Introduction}\label{sec:intro}
In this article we study a numerical scheme to 
solve linear parabolic problems, based on a weak space-time 
formulation. 
% WHAT ABOUT COMMAS IN THIS SENTENCE
The equation we consider, in its strong form and under assumptions that we 
specify in Section~\ref{sec:abstract_problem}, is
\begin{equation}\label{eq:HeatStrong}
\begin{aligned}
&\dot{u}(t) + Au(t) = f(t),\quad t\in(0,T],\\
&u(0) = u_0.
\end{aligned}
\end{equation}

% CLAIMING CENTRALITY AND INTRODUCING THE EXISTING FRAMEWORKS
During the last decades several authors have dealt with the space-time 
formulation of this problem. The main idea of a space-time formulation is 
to integrate the 
equation \eqref{eq:HeatStrong} in both the spatial and the temporal 
dimensions after multiplying it by a suitable space and time dependent test 
function. By doing the same with the initial 
condition, and by adding up the equations, we achieve the first space-time 
formulation of the problem, also called the primal 
formulation in other articles.

By means of a formal integration by parts of the 
term containing the time derivative, we achieve the 
weak space-time 
formulation of the problem, sometimes also called the second 
formulation (see \cite{Andreev1}, \cite{Andreev2}, 
\cite{Mollet}, \cite{SchwabSuli}, 
\cite{CheginiStev}) or natural formulation (see \cite{Fra}).
For both formulations, the main tool to prove the existence and 
uniqueness of the solution is the Banach--Ne{\v{c}}as--Babu{\v{s}}ka 
theorem, see Theorem~\ref{thm:bnbthm_abstract} below.

% REVIEWING ITEMS OF PREVIOUS RESEARCH
% NOTE TO YOURSELF: you should expand this section, reviewing and not just 
% listing 
Although such a theory was originally used to deal with 
mixed formulations of elliptic problems, from the late eighties it has also 
been used in connection 
with parabolic problems. A first analysis of numerics for evolution 
problems based on space-time 
formulations can be found in 
\cite{BabuskaJanik,BabuskaJanik2}.

In \cite{SchwabSte}, a discretization 
of evolution equations based on the primal formulation of the problem is 
discussed. The problem is restated as a bi-infinite matrix problem and 
discretized by an adaptive wavelet method. The proof of 
well-posedness 
of the abstract problem presented in the appendix of this article is of great 
relevance, since many other articles explicitly 
refer to it. In \cite{SchwabSuli} the second 
space-time formulation is deployed to construct adaptive numerical schemes; 
this choice allows the authors to apply the theory presented in previous paper 
to parabolic PDE's in infinite dimensions, where the solution is in 
general not regular enough to allow the use of the first space-time formulation

In \cite{CheginiStev}, the second 
space-time formulation is used to further investigate what was 
studied in \cite{SchwabSte}, under the extra assumption that the bi-infinite 
matrix system is truly sparse.

In \cite{AndreevThesis,Andreev1} the stability of space-time Petrov--Galerkin 
discretizations of the problem is studied for both the first and the 
second formulations. A possible selection of stable space-time trial and test 
spaces is presented, and a CFL condition is derived. Such a condition is shown 
to be necessary when trial and 
test spaces are chosen to be piecewise polynomials.
In \cite{Andreev2} the author proposes a Petrov--Galerkin space-time 
discretization of the heat equation on an unbounded time interval by 
means of Laguerre polynomials. Both the first and the second space-time 
formulations are investigated.

In \cite{Mollet} the author considers suitable hierarchical 
families of discrete spaces, both of finite element and 
wavelet type, and investigates the required number of 
extra layers in order to guarantee uniform boundedness of the discrete inf-sup 
constant in the second space-time formulation. 

In \cite{UrbanPatera2}, the second 
space-time formulation is used as a natural framework in which the reduced 
basis method can be investigated, allowing the authors to derive 
sharp a posteriori error bounds.  

% COUNTERCLAIMING
However, in all the works on the second space-time formulation, the
authors choose to neglect a term that naturally arises from the
integration by parts. This is achieved by using test functions which
vanish at the final time instant.  Although this is justified because
the neglected term is a pointwise version of the term which is kept,
the neglected term can play an important role, as noticed, for
example, in \cite{LarssonMolteni}, where the second space-time
formulation is used to study a stochastic variant of
\eqref{eq:HeatStrong}.

% ANNOUNCING PRINCIPAL FINDINGS
By keeping such a term in the current paper, not only do we have a
framwork for stochastic evolution equations, but we also obtain
estimates in the $L^{\infty}((0,T);H)$-norm in addition to the natural
$L^{2}((0,T);V)$-norm and we can construct a numerical scheme that is
superconvergent at the temporal mesh points.

% INDICATING STRUCTURE
The paper is structured as follows.  In Section~\ref{sec:abstract_problem} we 
present the abstract framework for the weak space-time formulation based on the 
Banach--Nec\v as--Babu\v ska ``inf-sup'' theorem.  
Section~\ref{sec:discretization} introduces the Petrov--Galerkin approximation 
based on piecewise polynomials in space and time.  The trial functions are 
discontinuous of degree $q\ge0$ in time while the test functions are continuous 
of degree $q+1$.  The possibility of extracting point values at the temporal 
nodes is emphasized.  Section~\ref{sec:error} is devoted to the \emph{a priori} error 
estimates based on quasi-optimality.  A CFL condition is required.  The temporal 
order in the natural norm is $q+1$.  However, we  note that the piecewise 
constant approximation ($q=0$) is of second order in time by a comparison with the 
Crank--Nicolson method.  In Section~\ref{sec:semidiscrete} we give a direct 
proof of this by showing that our method is actually superconvergent of order 
$2(q+1)$ at the temporal nodes.  The proof is based on separating the temporal 
and spatial error and a duality argument.  We only present the analysis of the 
temporally semidiscrete part.  The proof avoids the use of a CFL condition, 
which is not available for pure time discretizations.  The temporal convergence 
rates are demonstrated in numerical experiments in Section~\ref{sec:numerics}. 

%%%%%%%%%%%%%%%%%%%%%%%%%%%%%%%%%%%%%%%%%%%%%%%%%%%%%%%%%%%%%%%%%%%%%%%%%%%%%%%%
%%%%%%%%%%%%%%%%%%%%%%%%%%%%%%%%%%%%%%%%%%%%%%%%%%%%%%%%%%%%%%%%%%%%%%%%%%%%%%%%
%%%%%%%%%%%%%%%%%%%%%%%%%%%%%%%%%%%%%%%%%%%%%%%%%%%%%%%%%%%%%%%%%%%%%%%%%%%%%%%%
%%
%%
%% The problem in its abstract form
%%
%%
%%%%%%%%%%%%%%%%%%%%%%%%%%%%%%%%%%%%%%%%%%%%%%%%%%%%%%%%%%%%%%%%%%%%%%%%%%%%%%%%
%%%%%%%%%%%%%%%%%%%%%%%%%%%%%%%%%%%%%%%%%%%%%%%%%%%%%%%%%%%%%%%%%%%%%%%%%%%%%%%%
%%%%%%%%%%%%%%%%%%%%%%%%%%%%%%%%%%%%%%%%%%%%%%%%%%%%%%%%%%%%%%%%%%%%%%%%%%%%%%%%

\section{The abstract problem}\label{sec:abstract_problem}
%%%%%%%%%%%%%%%%%%%%%%%%%%%%%%%%%%%%%%%%%%%%%%%%%%%%%%%%%%%%%%%%%%%%%%%%%%
%%									%%
%%		Our setting						%%
%%									%%
%%%%%%%%%%%%%%%%%%%%%%%%%%%%%%%%%%%%%%%%%%%%%%%%%%%%%%%%%%%%%%%%%%%%%%%%%%
\subsection{An abstract framework}
We assume that a Gelfand triple $V \hookrightarrow H \hookrightarrow V^*$ is 
given, where $V$ and $H$ are separable Hilbert spaces such that $V$ is densely 
embedded into $H$. We assume that the operator $A$, which appears in 
\eqref{eq:HeatStrong}, is associated to a symmetric bilinear form 
$a(\cdot,\cdot)$ that 
satisfies the following conditions:
\begin{alignat}{3}
\tag{boundedness} & \abs{a(u,v)} \leq \AM \norm{V}{u}
\norm{V}{v},\quad
&& \, u,v\in V, \\
\tag{coercivity}  & a(v,v) \geq \Am \norm{V}{v}^2, \quad && \, v\in V,
\end{alignat}
for some positive constants $\AM$ and $\Am$. We introduce the Lebesgue-Bochner 
spaces
\begin{align*}
{\cY^t = L^2((0,t);V),} \quad {\cX^t = L^2((0,t);V) \cap 
H^1((0,t);V^*),}
\end{align*}
with norms defined by
\begin{equation*}%\label{eq:original_norms}
\begin{aligned}
\norm{\cY^t}{y}^2 &:= \norm{L^2((0,t);V)}{y}^2 = \int_0^{t}\norm{V}{y(s)}^2\,\dd
s,\\
\norm{\cX^t}{x}^2 &:= \norm{H}{x(0)}^2 + \norm{L^2((0,t);V)}{x}^2 + 
\norm{L^2((0,t);V^*)}{\dot{x}}^2.
\end{aligned}
\end{equation*}
We use the notation $ \cY^t_H$ for the space $\cY^t \times H$ endowed
with the product norm, and we use the convention that $\cY = \cY^T$,
$\cY_H = \cY \times H$, and $\cX=\cX^T$, when $t=T$. We recall that
the space $\cX^t$ is densely embedded in $ \sC([0,t];H)$, So that
pointwise values of $x\in\cX$ make sense. With the present choice of
norm the embedding constant is $1$; in particular, it does not depend
on $t$ or $V$, see \cite{LarssonMolteni}.
%%%%%%%%%%%%%%%%%%%%%%%%%%%%%%%%%%%%%%%%%%%%%%%%%%%%%%%%%%%%%%%%%%%%%%%%%%
%%									%%
%%		Out formulation						%%
%%									%%
%%%%%%%%%%%%%%%%%%%%%%%%%%%%%%%%%%%%%%%%%%%%%%%%%%%%%%%%%%%%%%%%%%%%%%%%%%
%%%%%%%%%%%%%%%%%%%%%%%%%%%%%%%%%%%%%%%%%%%%%%%%%%%%%%%%%%%%%%%%%%%%%%%%%%
%%									%%
%%		The first formulation					%%
%%									%%
%%%%%%%%%%%%%%%%%%%%%%%%%%%%%%%%%%%%%%%%%%%%%%%%%%%%%%%%%%%%%%%%%%%%%%%%%%

The first space-time formulation of 
\eqref{eq:HeatStrong} reads:
\begin{align}\label{eq:first_space_time}
u\in \cX: \sB(u,y) = \sF(y), \quad \forall y=(y_1,y_2)\in \cY_H.
\end{align}
Here we use the bilinear form:
\begin{equation*}
\begin{aligned}
&\sB \colon \cX \times \cY_H \rightarrow \bR,\\
&\sB(x,y) := \int_0^T{\dual{\dot{x} +A x}{y_1}}\,\dd s + 
\inner{x(0)}{y_2},
\end{aligned}
\end{equation*}
and the load functional
\begin{align*}
\sF \in \cY_H^*, \quad \sF(y):= \int_0^T{ \tdual{f}{y_1} }\,\dd t 
+\inner{u_0}{y_2}. 
\end{align*}
If we integrate by parts and swap the test and trial spaces, then we obtain 
the weak (or second) space-time formulation
\begin{equation} \label{eq:secondspacetime}
u=(u_1,u_2) \in \cY_H: \sB^*(u,x) = \sF(x), \quad \forall x 
\in \cX,
\end{equation}
where the bilinear form and the load functional are now:
\begin{equation}\label{eq:bilinear_form}
\begin{aligned}
&\sB^* \colon \cY_H \times \cX \rightarrow \bR,\\
&\sB^*(y,x) := \int_0^T{ \tdual{y_1}{-\dot{x} +Ax} }\,\dd s + 
\inner{y_2}{x(T)},
\end{aligned}
\end{equation}
\begin{equation}\label{eq:load_functional}
\begin{aligned}
&\sF \in \cX^*, \quad \sF(x):= \int_0^T{ \tdual{f}{x} }\,\dd t 
+\inner{u_0}{x(0)}.
\end{aligned}
\end{equation}
%%%%%%%%%%%%%%%%%%%%%%%%%%%%%%%%%%%%%%%%%%%%%%%%%%%%%%%%%%%%%%%%%%%%%%%%%%
%%									%%
%%		The second component					%%
%%									%%
%%%%%%%%%%%%%%%%%%%%%%%%%%%%%%%%%%%%%%%%%%%%%%%%%%%%%%%%%%%%%%%%%%%%%%%%%%
It is easy to see that the second component $u_2$ of the solution $u$ 
to 
\eqref{eq:secondspacetime} depends on the final time instant $T$. We 
can think of parametrizing \eqref{eq:secondspacetime} 
over $t\in[0,T]$ and reformulate it as a family of 
problems:
\begin{equation}\label{eq:abstract_problem_param_t}
u=(u_1,u_2(t)) \in \cY^t_H: \sB_{t}^*(u,x) = \sF_{t}(x), \quad 
\forall x 
\in \cX^t,
\end{equation}
where $\sF_{t}$ and $\sB_t^*$ are as before, but restricted to the
spaces $\cY^t_H$ and $\cX^t$.

If the right-hand side of \eqref{eq:HeatStrong} is regular enough, as
in \S~\ref{subsubsub:regular_RHS} below, then $u_1$ has a square
integrable weak derivative and therefore belongs to the space
$\cX\subset \sC([0,T];H)$, and $u_1=u_2$.
However, if the right-hand side is less regular, as in
\S~\ref{subsubsub:stoch_RHS} and \S~\ref{subsubsub:nowhere},  then
$u_1$ need not be differentiable nor continuous, but  $u_2$ is
a continuous time-dependent $H$-valued version of $u_1$:
\begin{equation*}%\label{eq:u1_version_u2}
\int_0^T \norm{H}{u_1(t)-u_2(t)}^2\,\dd t = 0.
\end{equation*}
% However, we find it convenient to keep $u_1$ and 
% $u_2$ separated, 
% since it allows us to use of this approach for a 
% broader family of right-hand sides, which includes, among other, a stochastic 
% integral with trace-class noise (see \cite{LarssonMolteni}).
The second component $u_2$ is often omitted in other works 
(e.g., \cite{SchwabSuli}, \cite{Mollet}), where the following weak space-time 
formulation is used:
\begin{equation*}%\label{eq:weak_space_time_homo}
u \in \cY: \sB^*(u,x) = \sF(x), \quad \forall x \in
\cX_{0,\{T\}}:=\{x\in\cX:x(T)=0\}.
\end{equation*}
We keep $u_2$ in order to be able to extract point values.  

%%%%%%%%%%%%%%%%%%%%%%%%%%%%%%%%%%%%%%%%%%%%%%%%%%%%%%%%%%%%%%%%%%%%%%%%%%
%%									%%
%%		THE RIGHT HAND SIDES					%%
%%									%%
%%%%%%%%%%%%%%%%%%%%%%%%%%%%%%%%%%%%%%%%%%%%%%%%%%%%%%%%%%%%%%%%%%%%%%%%%%
In order to appreciate the weak space-time formulation, we briefly 
recall the two main advantages that we want to exploit: larger variety of 
source terms and the possibility to obtain pointwise bounds.  

\subsection{A larger variety of right-hand sides}
First of all, the weak-space time formulation allows the use of a broad family 
of possible source terms.
\subsubsection{Regular right-hand side}\label{subsubsub:regular_RHS}
The basic case that we analyse is given by
\begin{equation}\label{eq:regular_rhs}
\sF_t (x)= \int_0^t{ \dual{f(s)}{x(s)} }\,\dd s + \inner{u_0}{x(0)}, \quad 
t\in[0,T],
\end{equation}
for some $f \in L^2((0,T);V^*)$ and $u_0 \in H$. In this case, we have $u_2 = 
u_1 \in \cX$.
Indeed, by taking $x \in \sC_0^{\infty}([0,t];V)$ in 
\eqref{eq:abstract_problem_param_t}, we obtain
\begin{equation*}
\int_0^t{ \inner{u_1(s)}{-\dot{x} (s) }}\,\dd s = \int_0^t{ \dual{f(s) - 
Au_1(s)}{x(s)} }\,\dd s.
\end{equation*}
Thus, $u_1$ has a weak derivative $\dot{u}_1 = f-Au_1 \in L^2((0,T);V^*)$, so 
that 
$u_1 \in \cX$. Then we can integrate by parts in 
\eqref{eq:abstract_problem_param_t} and conclude that $u_2 = u_1$ and that 
they both belong to $\cX \subset \sC([0,T];H)$.

\subsubsection{Piecewise regular right-hand side}\label{subsubsub:pw}
A more general case is offered by
\begin{equation*}%\label{eq:jump_RHS}
\sF_t(x) = \int_0^t{ \dual{f(s)}{x(s)} }\,\dd s + \sum_{t_i \leq 
t}\inner{\zeta_i}{x(t_i)}, \quad t\in[0,T],
\end{equation*}
for some $f \in L^2((0,T);V^*)$, $ \{\zeta_i\}_{i=1,\ldots,M} \in H $ and $ 
\{t_i\}_{i=1,\ldots,M} \subset [0,T] $.

In this case the conclusions presented above only hold piecewise. 
In particular, the values of $\zeta_i$ represent the 
jumps of the solution at time $t_i$.

\subsubsection{Stochastic integral}\label{subsubsub:stoch_RHS}
A more general example is represented by a functional which is defined
$\omega$-wise, for $\omega$ in a complete probability space
$(\Omega,\Sigma_t,\bP) $, and of the form $\sF_t + \sW_t$. Here
$\sF_t$ is as in \S~\ref{subsubsub:regular_RHS} and $\sW_t$ is a weak
stochastic integral with respect to an $H$-valued Wiener process $W$,
with operator-valued integrand $\Psi$:
\begin{equation}\label{eq:stoch_RHS}
\sW_t(x) = \int_0^t{ \inner{\Psi(s) \, \dd W(s)}{x(s)} }, \quad t\in[0,T].
\end{equation}
The details of such an equation have been presented in 
\cite{LarssonMolteni} and we refrain from recalling them here. It holds that 
$u_1$ and $u_2$ are versions of each other, in the sense that 
$ u_1 \in L^2((0,T);V)$, $u_2 \in \sC([0,T];H)$ almost surely and $u_1 = 
u_2$ in $L^2(\Omega 
\times (0,T);H)$. This case represents an important example in which the weak 
space-time 
formulation cannot be replaced by the first space-time 
formulation, since the 
Wiener process is nowhere differentiable and therefore $u_1 \notin \cX$.

\subsubsection{Nowhere differentiable right-hand side}\label{subsubsub:nowhere}
The most general type of right-hand side that we can handle has the form 
\begin{equation}\label{eq:non_smooth_RHS}
\sF_t(x) = \int_0^t{ \tdual{g(s)}{-\dot{x}(s)} }\,\dd s - \inner{g(t)}{x(t)} + 
\inner{g(0)}{x(0)}  , \quad t\in[0,T],
\end{equation}
for a function $g \in L^2((0,T);V) \cap \sC([0,T];H)$, and with $g$ nowhere 
differentiable, so that we are not in one of the first two cases in this 
list.

Similar conclusions to the ones obtained for the stochastic integral hold even 
in this case. We have that $u_1 \notin \cX$, that $u_1 = u_2$ in 
$L^2((0,T);H)$, and that $u_2 \in \sC([0,T];H)$. In case $g$ is smooth it is 
easy to see that integration by parts leads to a right-hand side of the same 
form 
as in \eqref{eq:regular_rhs}.

We want to stress that both in the case of a right-hand side of the form 
\eqref{eq:stoch_RHS} or \eqref{eq:non_smooth_RHS} the presence of $u_2$ is 
important, since point values $u_1(t)$ of $u_1$ are not well defined.

\subsection{Point values and decompositions}
Another important advantage offered by the weak formulation is that
the solution is not required to be continuous in its first component
$u_1$. This allows us to split the time interval and to solve local
problems, where information is passed from one time interval to the
next through $u_2(t)$, see \eqref{eq:splitting} and
\S~\ref{subs:decomposition} below.  This can be exploited even on a
discrete level, by solving problems with different spatial
discretizations on each time interval, since the passage of
information between two different intervals occurs only by means of
the second component of the solution, $u_2$. This ensures a
flexibility in the choice of the spatial grid, which could in
principle change at each interval and still not cause any sort of
variational crime, since the discrete spaces would still be proper
subspaces of the continuous ones.
% {\color{red} This idea can be used 
% as starting point for deriving a posteriori analysis, both in the 
% case of a fixed spatial grid and in the case of a possibly time-dependent 
% spatial grid. In particular this can be of relevance for the 
% construction of adaptive algorithms.}
%%%%%%%%%%%%%%%%%%%%%%%%%%%%%%%%%%%%%%%%%%%%%%%%%%%%%%%%%%%%%%%%%%%%%%%%%%
%%									%%
%%		INF SUP_						%%
%%									%%
%%%%%%%%%%%%%%%%%%%%%%%%%%%%%%%%%%%%%%%%%%%%%%%%%%%%%%%%%%%%%%%%%%%%%%%%%%

\subsection{The inf-sup theorem} \label{sec:inf_sup_opt}
We recall the following theorem (see \cite{BabuskaAziz, Ern}):
\begin{theorem}[Banach--Ne{\v{c}}as--Babu{\v{s}}ka 
(BNB)]\label{thm:bnbthm_abstract}
Let $V$ and $W$ be Hilbert spaces. Given a bilinear form
$\sB \colon W \times V \rightarrow \bR $, such that
\begin{align}
\tag{BDD} \label{BDD}
C_B &:=  \sup_{0 \neq w \in W}\sup_{0 \neq v \in
V}\frac{\sB(w,v)}{\norm{W}{w}\norm{V}{v}} < \infty,
\end{align}
the associated linear operator $B \colon W \rightarrow V^*$, defined by
\begin{align*}
\duals{Bw}{v}{V^*}{V} := \sB(w,v) , \ 
\forall w \in W, \forall v \in V,
\end{align*}
is boundedly invertible if and only if the following two conditions are 
satisfied: 
\begin{align}
\tag{BNB1} \label{BNB1A}
c_B &:= \inf_{0 \neq w \in W}\sup_{0 \neq v \in
V}\frac{\sB(w,v)}{\norm{W}{w}\norm{V}{v}} >0,\\
\tag{BNB2} \label{BNB2A}
\forall v \in &V, \quad \sup_{0 \neq w \in W}\sB(w,v) >0.
\end{align}
\end{theorem}
The constant $c_B$ is called the {\it inf-sup constant}, while the constant   
$C_B$ is called the {\it boundedness constant}. Since $c_B^{-1} = 
\norm{\sL(V^*,W)}{B^{-1}} = \norm{\sL(W^*,V)}{(B^*)^{-1}} $, it follows 
that \eqref{BNB1A}--\eqref{BNB2A} are equivalent to
\begin{align}\label{eq:equivalent_to_bnb1-2}
\inf_{0 \neq w \in W}\sup_{0 \neq v \in
V}\frac{\sB(w,v)}{\norm{W}{w}\norm{V}{v}} = \inf_{0 \neq v \in V}\sup_{0 \neq w
\in W}\frac{\sB(w,v)}{\norm{W}{w}\norm{V}{v}}>0.
\end{align}
This allows to swap the spaces where the infimum and the supremum are 
taken. 

We now have to show that $\sB_t^*$ in \eqref{eq:bilinear_form} satisfies 
the assumptions of the BNB theorem on 
the spaces $\cY^t_H$ and $\cX^t$. The proof follows the same line as the one 
presented \cite{SchwabSte}; we omit the proof of the \eqref{BNB2A} since it 
does not contain any quantitative information. In order to obtain sharper 
bounds for $C_B$ and $c_B$, we introduce equivalent norms. This is of 
particular relevance in this new formulation, since we 
want to have a constant $1$ in front of the pointwise term $u_2$, in order to 
exploit the temporal decomposition, which we present in the next section.

In virtue of the properties of $A$, fractional powers are well defined
and the norms of $V$ and $V^*$ are equivalent to
$\norm{H}{A^{\frac12} \cdot}$ and $\norm{H}{A^{-\frac12} \cdot}$, respectively. For
a more detailed explanation of this fact we refer to \cite{Cioica}. We
therefore introduce equivalent norms on $\cX^t$ and $\cY^t_H$,
respectively, as follows:
\begin{equation*}%\label{eq:tilde_norms}
\begin{aligned}
\aabs{\cX^t}{x}^2 &:= \norm{H}{x(0)}^2 +  \int_0^t{ \Big(\norm{H}{ 
A^{\frac12} x(s)}^2 + 
\norm{H}{A^{-\frac12} \dot{x}(s)}^2\Big)}\,\dd s,\\
\aabs{\cY^t_H}{y}^2 &:= \norm{H}{y_2}^2 + \int_0^t{ \norm{H}{ A^{\frac12}
y_1(s)}^2}\,\dd s.
\end{aligned}
\end{equation*}

\begin{lemma}\label{lemma:equiv_norm}
The norm $\nnnorm{\cX^t}{\cdot}$, defined by
\begin{align*}
\nnorm{\cX^t}{x}:= \norm{H}{x(t)}^2 + 
\int_0^{t} \norm{H}{A^{\frac12} x(s) -A^{-\frac12}\dot{x}(s)}^2\,\dd s ,
\end{align*}
is equal to the norm $\aabs{\cX^t}{\cdot}$, for every $t\in[0,T]$. 
\end{lemma}
\begin{proof}
We have
\begin{align*}
\nnorm{\cX^t}{x} &= \norm{H}{x(t)}^2 + 
\int_0^{t}\Big(\norm{H}{A^{\frac12}x(s)}^2 
+\norm{H}{A^{-\frac12}\dot{x}(s)}^2 - 2\tdual{x(s)}{\dot{x}(s)}
\Big)\,\dd s\\
&= \norm{H}{x(0)}^2 + 
\int_0^{t}\Big(\norm{H}{A^{\frac12}x(s)}^2 
+\norm{H}{A^{-\frac12}\dot{x}(s)}^2 \Big)\,\dd s = \aabs{\cX^t}{x}^2,
\end{align*}
because $\inner{A^{\frac12}x(s) }{A^{-\frac12}\dot{x}(s)} = 
\tdual{x(s) }{\dot{x}(s)} = \frac12 \frac{\dd}{\dd t} \norm{H}{x(s)}^2$.
\end{proof}

%%%%%%%%%%%%%%%%%%%%%%%%%%%%%%%%%%%%%%%%%%%%%%%%%%%%%%%%%%%%%%%%%%%%%%%%%%
%%									%%
%%		consequences OF THE BNB THEOREM				%%	
%%									%%
%%%%%%%%%%%%%%%%%%%%%%%%%%%%%%%%%%%%%%%%%%%%%%%%%%%%%%%%%%%%%%%%%%%%%%%%%%

We now compute $c_B$ and $C_B$ for $\sB_t^*$ with respect 
to $\aabs{\cY^t_H}{\cdot}$ and $\aabs{\cX^t}{\cdot}$.
\begin{theorem}\label{thm:bnbthm}
The bilinear form $\sB_t^*(\cdot,\cdot)$ satisfies the following:
\begin{align}
&C_B:=\sup_{0\neq y \in \cY^t \times H}\sup_{0\neq x \in \cX^t} 
\frac{\sB_t^*(y,x)}{\aabs{\cY^t_H}{y} \aabs{\cX^t}{x}  } = 1, \label{eq:bdd}\\
&c_B:=\inf_{0\neq y \in \cY^t \times H} \sup_{0\neq x \in \cX^t}
\frac{\sB_t^*(y,x)}{\aabs{\cY^t_H}{y} \aabs{\cX^t}{x}  } = 1.\label{eq:infsup}
\end{align}
\end{theorem}
\begin{proof}
We first notice that
\begin{align*}
\abs{\sB_t^*(y,x)} &\leq \int_0^t{ \abs{ 
\tdual{y_1(s)}{-\dot{x}(s) + A x(s)} }   } \,\dd s + 
\abs{\inner{y_2}{x(t)}}\\
%%%%%%%%%%%%%%%%%%%%%%%%%%%%%%%%%%%%%%%%%%%%%%%%%%%%%%%%%%%%%%%%%5
&\leq \int_0^t{ \norm{H}{A^{\frac12}y_1(s)}\,
\norm{H}{ -A^{-\frac12}\dot{x}(s) + A^{\frac12}{x}(s)}  }\,\dd s 
+ 
\norm{H}{y_2}\norm{H}{x(t)}\\
&\leq \aabs{\cY^t_H}{y} \nnnorm{\cX^t}{x} = \aabs{\cY^t_H}{y} \aabs{\cX^t}{x} .
\end{align*}
This proves $C_B\leq 1$. To show $c_B\ge1$, we use the second 
variant in \eqref{eq:equivalent_to_bnb1-2} and prove
\begin{equation*}
\forall x \in \cX^t, \ \exists y_x \in \cY^t_H: \sB_t^*(y_x,x) \geq 
\aabs{\cY^t_H}{y_x}\nnnorm{\cX^t}{x} = \aabs{\cY^t_H}{y_x}\aabs{\cX^t}{x}.
\end{equation*}
For $x \in \cX^t$ we choose $y_x = \big( x - A^{-1}\dot{x}, x(t) \big)$, which 
belongs to $\cY^t_H$, since
\begin{align}\label{eq:clever_estimate_BNB}
\aabs{\cY^t_H}{y_x}^2 &= \norm{H}{ 
x(t)}^2+  \norm{L^2((0,t);H)}{A^{\frac12}(x - 
A^{-1}\dot{x})}^2  = \nnorm{\cX^t}{x} = \aabs{\cX^t}{x}^2.
\end{align}
By expanding the bilinear form and
using\eqref{eq:clever_estimate_BNB}, we have
\begin{align*}
\sB_t^*(y_x,x) &= \int_0^t \inner{ x(s) - A^{-1}\dot{x}(s)}{-\dot{x}(s)+ 
Ax(s)} \,\dd s + \norm{H}{x(t)}^2\\
&= \int_0^t \norm{H}{ A^{\frac12} x(s) - 
A^{-\frac12} \dot{x}(s)}^2 \,\dd s + \norm{H}{x(t)}^2
= \nnorm{\cX^t}{x} = \aabs{\cX^t}{x}\aabs{\cY^t_H}{y_x}.
\end{align*}
Hence, $c_B\ge1$. Since $c_B \leq C_B $, we conclude that they are both equal to $1$.
\end{proof}

%%%%%%%%%%%%%%%%%%%%%%%%%%%%%%%%%%%%%%%%%%%%%%%%%%%%%%%%%%%%%%%%%%%%%%%%%%%%%%%
%%									     %%
%%			consequence					     %%
%%									     %%
%%%%%%%%%%%%%%%%%%%%%%%%%%%%%%%%%%%%%%%%%%%%%%%%%%%%%%%%%%%%%%%%%%%%%%%%%%%%%%%
As a consequence, since the bilinear form fulfils the hypothesis 
of the the BNB theorem, the operator $B_t \in 
\cL(\cY^t_{H},(\cX^t)^*)$ associated with $
\sB_t^*(\cdot,\cdot)$ via $$ \sB_t^*(y,x) = \duals{B_t 
y}{x}{(\cX^t)^*}{\cX^t}$$ is boundedly
invertible, and $ \aabs{\cY^t_H}{y} \leq \norm{(\cX^t, 
\aabs{\cX^t}{\cdot})^*}{\sF}$.
%%%%%%%%%%%%%%%%%%%%%%%%%%%%%%%%%%%%%%%%%%%%%%%%%%%%%%%%%%%%%%%%%%%%%%
%%%%%%%%%%%%%%%%%%%%%%%%%%%%%%%%%%%%%%%%%%%%%%%%%%%%%%%%%%%%%%%%%%%%%%
%%%%%%%%%%%%%%%%%%%%%%%%%%%%%%%%%%%%%%%%%%%%%%%%%%%%%%%%%%%%%%%%%%%%%%
%%
%%
%%			 R H S 
%%
%%
%%%%%%%%%%%%%%%%%%%%%%%%%%%%%%%%%%%%%%%%%%%%%%%%%%%%%%%%%%%%%%%%%%%%%%
%%%%%%%%%%%%%%%%%%%%%%%%%%%%%%%%%%%%%%%%%%%%%%%%%%%%%%%%%%%%%%%%%%%%%%
%%%%%%%%%%%%%%%%%%%%%%%%%%%%%%%%%%%%%%%%%%%%%%%%%%%%%%%%%%%%%%%%%%%%%%
We note that for a right-hand side of the form \S~\ref{subsubsub:regular_RHS},
$\sF_t$ belongs to the dual space of $(\cX^t,\aabs{\cX^t}{\cdot})$ for any 
$t\leq T$, if 
$f\in 
L^2((0,T);V^*)$ and $u_0 \in H$. In fact,
\begin{equation}\label{eq:norm_rhs}
 \begin{aligned}
\norm{(\cX^t,\aabs{\cX^t}{\cdot})*}{\sF_t} &\leq \Big[ \int_0^t{ 
\norm{H}{A^{-\frac12}f(s)}^2 
}\,\dd 
s + 
\norm{H}{u_0}^2 \Big]^{\frac12} \\
&\leq \Big[ \Am^{-1} \int_0^t{ 
\norm{V^*}{f(s)}^2 
}\,\dd 
s + 
\norm{H}{u_0}^2 \Big]^{\frac12}.
\end{aligned}
\end{equation}
By combining the BNB theorem with \eqref{eq:norm_rhs}, we thus achieve 
the estimate
\begin{align*}
\int_0^t{ \norm{H}{A^{\frac12} u_1(s)}^2 }\,\dd s + \norm{H}{u_2(t)}^2  \leq 
\int_0^t{ \norm{H}{A^{-\frac12} f(s)}^2 }\,\dd s + \norm{H}{u_0}^2.
\end{align*}
In particular, by using the 
equivalence between $ \aabs{\cY^t_H}{\cdot}$ and $ 
\norm{\cY^t_H}{\cdot}$, and the last bound in \eqref{eq:norm_rhs},
we obtain that:
\begin{align}\label{eq:bound_norm_sol_V}
\Am \int_0^t{ \norm{V}{u_1(s)}^2 }\,\dd s + \norm{H}{u_2(t)}^2  
\leq  
\Am^{-1}\int_0^t{ \norm{V^*}{f(s)}^2 }\,\dd s + \norm{H}{u_0}^2.
\end{align}

We emphasize that we have a constant $1$ in front of $u_2$. Therefore, we 
can split and 
recompose the problem as we please, and the bounds for the norms will compose 
accordingly, without accumulation of constants. More precisely, if we consider 
the 
same problem on $[0,r]$ with initial data $u_0 \in H$, and on $[r,t]$ 
with initial data given by the $u_2(r) \in H$ previously obtained, then we 
have the two local bounds:
\begin{align}  \label{eq:splitting}
  \begin{aligned}
\Am \int_{0}^{r}\norm{V}{u_1(s)}^2\,\dd s + \norm{H}{u_2(r)}^2 &\leq    
\Am^{-1}\int_{0}^{r}\norm{V^*}{f(s)}^2\,\dd s + \norm{H}{u_0}^2 ,\\
\Am\int_{r}^{t}\norm{V}{u_1(s)}^2\,\dd s + \norm{H}{u_2(t)}^2 &\leq  
\Am^{-1}\int_{r}^{t}\norm{V^*}{f(s)}^2\,\dd s + \norm{H}{u_2(r)}^2 ,
  \end{aligned}
\end{align}
which sum up to the global bound \eqref{eq:bound_norm_sol_V}. 
We have thus a local inf-sup theory consistent with the global one, which can 
be exploited to derive local estimates which can be put together to build 
global estimates.

We summarize this in the following theorem.
\begin{theorem}[Existence and uniqueness] \label{thm:maintheorem2}
For a right-hand side of the form \S~\ref{subsubsub:regular_RHS}, with $u_0 \in 
H$ 
and $f \in L^2((0,T);V^*)$, there exists a unique solution $ u =(u_1,u_2)$ 
in $L^2((0,T);V)
\times \sC([0,T];H)$ to Problem~\eqref{eq:abstract_problem_param_t}. Its norm 
satisfies the following bound:
\begin{align*}
&\Am \int_0^T{\norm{V}{u_1(t)}^2}\,\dd t +
\sup_{t\in[0,T]}\norm{H}{u_2(t)}^2  \leq 
 \Am^{-1} \int_0^T{\norm{V^*}{f(t)}^2}\,\dd t +
\norm{H}{u_0}^2,
\end{align*}
and in particular it holds that $u_1 = u_2 \in \cX$.
\end{theorem}
\begin{remark}
In case the right-hand side is not the one introduced in 
\S~\ref{subsubsub:regular_RHS}, we still obtain existence and uniqueness as in 
Theorem~\ref{thm:maintheorem2}, but the bounds 
of the norms are modified according to the bounds that can be obtained for 
$\norm{(\cX^t, \aabs{\cX^t}{\cdot})^*}{\sF_t}$. The modifications for the cases 
presented in 
\S~\ref{subsubsub:pw} or \S~\ref{subsubsub:nowhere} are easy to derive, while 
for the 
case of \S~\ref{subsubsub:stoch_RHS} the theory required is more involved and 
we 
refer to \cite{LarssonMolteni} for the details.
\end{remark}

\subsection{Further spatial regularity}\label{subsec:spatial_reg}
In order to measure spatial regularity use the spaces
  $\dot{H}^{\gamma}=D(A^{\frac{\gamma}{2}})$ with norms
  $\norm{\dot{H}^{\gamma}}{v}=\norm{H}{A^{\frac{\gamma}{2}}v}$ for
  $\gamma\in\bR$.

\begin{theorem}[Spatial regularity] \label{maintheorem2_beta}
  Assume $\beta\ge0$. The bilinear
  form defining problem \eqref{eq:abstract_problem_param_t} is bounded
  and satisfies the inf-sup conditions on the couple of spaces
  $L^2((0,t);\dot{H}^{\beta+1}) \times \dot{H}^{\beta}$ and 
  $L^2((0,t);\dot{H}^{1-\beta})\cap H^1((0,t);\dot{H}^{-1-\beta}) $.
  In particular, for a right-hand side of the form
  \S~\ref{subsubsub:regular_RHS}, if $f \in L^2((0,T);\dot{H}^{\beta-1})$ and 
$u_0 \in \dot{H}^{\beta}$, there exists a unique solution $ u =(u_1,u_2) \in 
L^2((0,T);\dot{H}^{\beta+1}) \times
  \sC([0,T];\dot{H}^{\beta})$
  to \eqref{eq:abstract_problem_param_t}. Its norm satisfies the
  following bound:
\begin{align*}
&\int_0^T{\norm{\dot{H}^{\beta+1}}{u_1(t)}^2}\,\dd t +
\sup_{t\in[0,T]}\norm{\dot{H}^{\beta}}{u_2(t)}^2  \leq 
\int_0^T{\norm{\dot{H}^{\beta-1}}{f(t)}^2}\,\dd t +
\norm{\dot{H}^{\beta}}{u_0}^2,
\end{align*}
and it holds, in particular, that 
$u_1 = u_2 \in L^2((0,T);\dot{H}^{\beta+1})\cap H^1((0,T);\dot{H}^{\beta-1})$. 
\end{theorem}

%%%%%%%%%%%%%%%%%%%%%%%%%%%%%%%%%%%%%%%%%%%%%%%%%%%%%%%%%%%%%%%%%%%%%%%%%%%%
%%%%%%%%%%%%%%%%%%%%%%%%%%%%%%%%%%%%%%%%%%%%%%%%%%%%%%%%%%%%%%%%%%%%%%%%%%%%
%%%%%%%%%%%%%%%%%%%%%%%%%%%%%%%%%%%%%%%%%%%%%%%%%%%%%%%%%%%%%%%%%%%%%%%%%%%%
%%
%%
%%  			D I S C R E T I Z A T I O N
%%
%%
%%%%%%%%%%%%%%%%%%%%%%%%%%%%%%%%%%%%%%%%%%%%%%%%%%%%%%%%%%%%%%%%%%%%%%%%%%%%
%%%%%%%%%%%%%%%%%%%%%%%%%%%%%%%%%%%%%%%%%%%%%%%%%%%%%%%%%%%%%%%%%%%%%%%%%%%%
%%%%%%%%%%%%%%%%%%%%%%%%%%%%%%%%%%%%%%%%%%%%%%%%%%%%%%%%%%%%%%%%%%%%%%%%%%%%

\section{Discretization}\label{sec:discretization}
%%%%%%%%%%%%%%%%%%%%%%%%%%%%%%%%%%%%%%%%%%%%%%%%%%%%%%%%%%%%%%%%%%%%%%%%%%
%%									%%
%%		Assumptions						%%	
%%									%%
%%%%%%%%%%%%%%%%%%%%%%%%%%%%%%%%%%%%%%%%%%%%%%%%%%%%%%%%%%%%%%%%%%%%%%%%%%
We start this section by introducing a discretization based on test functions 
which are piecewise linear in time and trial functions which are piecewise 
constant in time. The scheme that we obtain turns out to be a modification of 
the Crank--Nicolson scheme, namely with a first step of Euler backward and a 
final step of Euler forward.

\subsection{Discretization with polynomials of lowest degree in 
time}\label{subsec:discretization_low}
%%%%%%%%%%%%%%%%%%%%%%%%%%%%%%%%%%%%%%%%%%%%%%%%%%%%%%%%%%%%%%%%%%%%%%%%%%
%%									%%
%%		Discretization						%%	
%%									%%
%%%%%%%%%%%%%%%%%%%%%%%%%%%%%%%%%%%%%%%%%%%%%%%%%%%%%%%%%%%%%%%%%%%%%%%%%%
We consider a partition of the time interval $[0,T]$, given by
$\cT_k = \{ 0=t_0<\dots <t_i<t_{i+1}<\dots <t_N=T\}$, with $k_i = t_{i+1}-t_{i}$,
and $k = \max_i {k_i}$.  We denote by $\cT^n_k$ the partition
$\cT_k$ restricted to the interval $[0,t_n]$.  We denote by $I_i$ the
interval $[t_{i},t_{i+1}]$, by $S_k$ the space of continuous piecewise
linear functions with respect to $ \cT_k$, and by $Q_k$ the space of
piecewise constant functions for the same partition, with the
convention that $S_k^n$ and $Q_k^n$ refer to the partition
$\cT^n_k$. We introduce $V_h$ as a standard finite element space of
continuous piecewise polynomials of degree less or equal to $p$, over
a quasi-uniform family of triangulations of the spatial domain, with
mesh size $h$. Since temporal discretization is our main concern, we
assume that $p\geq 1$ is sufficiently large for our analysis to make
sense.

The finite-dimensional subspaces that we use are defined as
$\cY_{h,k} := Q_k \otimes 
V_h $, and $ \cX_{h,k} := S_k
\otimes V_h$; consistently with the notation introduced above we 
introduce the 
family of spaces $\cY^n_{h,k}$ and $\cX^n_{h,k}$.

We denote the standard basis of piecewise linear ``hat'' functions generating 
$S_k$ 
by $\{\phi_i\}_{i=0}^N$ and the standard basis of piecewise constant functions 
generating $Q_k$ by $\{\psi_i\}_{i=0}^{N-1}$. We denote by 
$\sB_{}^*$ and $\sF_{}$ the bilinear form and the load functional 
defined in \eqref{eq:bilinear_form}--\eqref{eq:load_functional}. If we start 
from the formulation in 
\eqref{eq:secondspacetime}, then the discretized 
problem can be written as:
\begin{equation} \label{eq:discrete_problem}
U \in {\cY_{h,k} \times V_h}\colon \sB_{}^*(U,X) = \sF_{}(X), 
\quad \forall X \in
{\cX_{h,k}}.
\end{equation}
%%%%%%%%%%%%%%%%%%%%%%%%%%%%%%%%%%%%%%%%%%%%%%%%%%%%%%%%%%%%%%%%%%%%%%%%%%
%%									%%
%%		Existence, uniqueness and new norm			%%	
%%									%%
%%%%%%%%%%%%%%%%%%%%%%%%%%%%%%%%%%%%%%%%%%%%%%%%%%%%%%%%%%%%%%%%%%%%%%%%%%

For a formal proof of the existence and uniqueness of a solution to the 
discrete problem in \eqref{eq:discrete_problem}, we follow \cite{UrbanPatera2}, 
where the authors show that the inf-sup condition holds, and that the 
discrete 
inf-sup constant is the same as the inf-sup-constant obtained in the 
continuous case. However, in 
order to do so, the space $\cX_{h,k}$ is endowed with a different norm, 
depending on 
the discretization:
\begin{align*}
\norm{\cX_k}{X}^2&:=
\norm{H}{X(0)}^2+\sum_{i=0}^{N-1}{\int_{I_i}{\Big(\norm{V^*}{\dot{X}}^2 +
\norm{V}{\Pi_i X}^2 \Big) }\,\dd s },
\end{align*}
and similarly
\begin{align*}
\aabs{\cX_k}{X}^2&:= \norm{H}{X(0)}^2+
\sum_{i=0}^{N-1}{\int_{I_i}{\Big(\norm{H}{A^{-\frac12}\dot{X}}^2 +
\norm{H}{A^{\frac12}\Pi_i X}^2 \Big) }\,\dd s },
\end{align*}
where $\Pi$ is the orthogonal projection, defined locally by 
$\big(\Pi_i X \big)(t)=\frac{1}{k_i}\int_{I_i}{X(s)}\,\dd s$, $t\in I_i$. 

We can now repeat the argument 
of Lemma~\ref{lemma:equiv_norm} and Theorem~\ref{thm:bnbthm} in 
$(\cY_{h,k} \times V_h, \aabs{\cY_H}{\cdot})$ and $(\cX_{h,k} , 
\aabs{\cX_k}{\cdot})$, and obtain inf-sup constant $c_B=1$ and boundedness 
constant $C_B=1$ (cf.~Lemma~\ref{lemma:equiv_norm_sd} and 
Theorem~\ref{thm:bnbthm_sd} below). What remains now is to bound $\sF$ with 
respect to 
the 
modified norm $\aabs{\cX_k}{\cdot}$ instead of $\aabs{\cX}{\cdot}$. 
Comparing the two norms, we note that, for all $X 
\in \cX_{h,k}$,
\begin{align*}%\label{eq:definition_cs}
\sum_{i=0}^{N-1}{\int_{I_i}{\Big(\norm{V^*}{\dot{X}}^2 +
\norm{V}{X}^2 \Big) }\,\dd s 
} \leq c_{\rm S}^2 \sum_{i=0}^{N-1}{\int_{I_i}{\Big(\norm{V^*}{\dot{X}}^2 +
\norm{V}{\Pi_i X}^2 \Big) }\,\dd s },
\end{align*}
since $\cX_{h,k}$ is finite-dimensional, and 
where $c_{\rm S}$ is in general not uniform in the choice of the spaces. This leads 
to the equivalence of norms:
\begin{align}\label{eq:equivalence_norm_cX}
 \aabs{\cX_k}{X} \leq  \aabs{\cX}{X} \leq \max(1,c_{\rm S})\aabs{\cX_k}{X}, \quad 
x\in\cX_{h,k}.
\end{align}

The discrete 
problem is therefore not stable with respect to the original norms, unless 
something more is assumed on $c_{\rm S}$. In 
\cite{AndreevThesis} it was shown that a 
sufficient condition for the uniform boundedness of $c_{\rm S}$ is:
\begin{align}\label{eq:CFL}
C_{\rm CFL} := k \sup_{v \in V_h}{\frac{ \norm{V}{v} }{ 
\norm{V^*}{v}}}  < \infty, \qquad \mbox{for all $h$ and $k$}.
\end{align}
By quasi-uniformity and an inverse inequality, this reduces to a CFL condition 
$ k \leq C h^2$. Thus \eqref{eq:CFL} ensures the stability of the discrete 
problem with respect to the original 
norms. More precisely, for a right-hand side of the 
form \S~\ref{subsubsub:regular_RHS} we have, similarly to \eqref{eq:norm_rhs}:
\begin{align}\label{eq:RHS_discrete}
\norm{(\cX_{h,k}, \aabs{\cX_k}{\cdot})^*}{\sF} \leq \Big( c_{\rm S}^2 
\Am^{-1}\norm{L^2((0,T);V^*)}{f}^2 + 
\norm{H}{u_0}^2\Big)^{\frac12}.
\end{align}
Within this setting, an analogue of Theorem~\ref{thm:maintheorem2} holds 
for  $U\in(\cY_{h,k} \times V_h,\norm{\cY_H}{\cdot})$ with the bound
modified as in \eqref{eq:RHS_discrete}. 

In order to see that \eqref{eq:discrete_problem} amounts to a time-stepping 
scheme, we introduce the following notation:
\begin{align*}
&F_i^L := \frac{2}{k_{i-1}}\int_{t_{i-1}}^{t_{i}}{f\phi_i}\,\dd s,  
&&F_i^R:=\frac{2}{k_{i}}\int_{t_{i}}^{t_{i+1}}{f\phi_i}\,\dd s, \\
&\innernosym{A_i^L u}{v} := \frac{2}{k_{i-1}}\int_{t_{i-1}}^{t_{i}}{ 
a(u,v \otimes \phi_i)}\,\dd s,   
&&\innernosym{A_i^R u}{v}:=\frac{2}{k_{i}}\int_{t_{i}}^{t_{i+1}}{a(u,v 
\otimes \phi_i)}\,\dd s.
\end{align*}
The discrete problem, on the pair of spaces $(\cY_{h,k} \times 
V_h,\cX_{h,k})$, can be written 
explicitly as follows, for any $v \in V_h$:
\begin{align*}
\innernosym{U_1^{(0)} - u_0}{v} + \frac{k_0}{2}\innernosym{A_0^R U_1^{(0)}}{v} 
&=
\frac{k_0}{2} \innernosym{F_0^{R}}{v},\\
\innernosym{ U_1^{(i)} - U_1^{(i-1)}}{v} + \frac12\innernosym{{ k_i A^R_{i} 
U_1^{(i)} + 
k_{i-1}A^L_{i}U_1^{(i-1)}}}{v} &=  
\frac12\innernosym{ { k_{i}F_{i}^R  + k_{i-1} F_i^L } }{v},\\
\innernosym{U_2^{(N)} - U_1^{(N-1)}}{v} + \frac{k_{N-1}}{2} 
\innernosym{A_N^L U_1^{(N-1)}}{v}
&= \frac{k_{N-1}}{2}\innernosym{F_N^L}{v}.
\end{align*}
Here the $U_1^{(i)} \in V_h$ denote the coefficients of $U_1 = 
\sum_{i=0}^{N-1} U^{(i)} \psi_i$, and $U_2^{(N)} \in V_h$ is the approximation 
of 
$u_2(t_N)$. The scheme is a combination of one step of backward Euler, 
several steps of Crank--Nicolson, and a final step of forward Euler. 

From the discrete counterpart to equation \eqref{eq:bound_norm_sol_V}
and from \eqref{eq:RHS_discrete}, it follows that the norm of the
numerical solution is bounded as follows:
\begin{equation}\label{eq:norm_sol_discrete}
\begin{aligned}
\Am \norm{\cY}{U_1}^2 + \norm{H}{U_2^{(N)}}^2  &\leq
c_{\rm S}^2\Am^{-1} \norm{L^2((0,T);V^*)}{f}^2 + \norm{H}{u_0}^2 .
\end{aligned}
\end{equation}

%%%%%%%%%%%%%%%%%%%%%%%%%%%%%%%%%%%%%%%%%%%%%%%%%%%%%%%%%%%%%%%%%%%%%%%%%%%%%%%
%%%%%%%%%%%%%%%%%%%%%%%%%%%%%%%%%%%%%%%%%%%%%%%%%%%%%%%%%%%%%%%%%%%%%%%%%%%%%%%
%%%%%%%%%%%%%%%%%%%%%%%%%%%%%%%%%%%%%%%%%%%%%%%%%%%%%%%%%%%%%%%%%%%%%%%%%%%%%%%
%%
%%
%%	 D E C O M P O S I T I O N 
%%
%%
%%%%%%%%%%%%%%%%%%%%%%%%%%%%%%%%%%%%%%%%%%%%%%%%%%%%%%%%%%%%%%%%%%%%%%%%%%%%%%%
%%%%%%%%%%%%%%%%%%%%%%%%%%%%%%%%%%%%%%%%%%%%%%%%%%%%%%%%%%%%%%%%%%%%%%%%%%%%%%%
%%%%%%%%%%%%%%%%%%%%%%%%%%%%%%%%%%%%%%%%%%%%%%%%%%%%%%%%%%%%%%%%%%%%%%%%%%%%%%%
\subsection{Decomposition of the scheme}\label{subs:decomposition}
By noticing that in the case of a partition 
with a single element, the scheme reduces to
\begin{align*}
\innernosym{U_1^{(0)} - u_0}{v} + \frac{k_0}{2}\innernosym{A_0^R U_1^{(0)}}{v} 
&=
\frac{k_0}{2} \innernosym{F_0^{R}}{v},\\
\innernosym{U_2^{(1)} - U_1^{(0)}}{v} + \frac{k_0}{2}\innernosym{A_1^L 
U_1^{(0)}}{v} 
&= \frac{k_0}{2}\innernosym{F_1^L}{v}, 
\end{align*}
we can think of iterating such a decomposition over each time interval
$I_i$, thus obtaining the extra values that approximate $u_2(t_i)$ at
each grid point $t_i$.

The scheme becomes, for $i=0,\ldots,N-1$ and $U^{(0)} = u_0$:
\begin{equation}\label{eq:the_scheme}
\begin{aligned}
\innernosym{U_1^{(i)} - U_2^{(i)}}{v} + 
\frac{k_i}{2}\innernosym{A_i^R U_1^{(i)}}{v} 
&= 
\frac{k_{i}}{2}\innernosym{F_i^{R}}{v},\\
\innernosym{U_2^{(i+1)} - U_1^{(i)}}{v} + 
\frac{k_i}{2}\innernosym{A_{i+1}^L U_1^{(i)}}{v} 
&= 
\frac{k_{i}}{2}\innernosym{F_{i+1}^{L}}{v}.
\end{aligned}
\end{equation}
It follows from a suitable variant of Theorem~\ref{thm:maintheorem2} that the 
following 
holds: 
\begin{equation*}%\label{eq:norm_sol_discrete_2}
\begin{aligned}
\Am \norm{\cY}{U_1}^2 + \max_{i=1,\ldots,N}\norm{H}{U_2^{(i)}}^2  &\leq 
c_{\rm S}^2\Am^{-1}\norm{L^2((0,T);V^*)}{f}^2 + \norm{H}{u_0}^2 .
\end{aligned}
\end{equation*}
\begin{remark}
An important thing to notice is that $U_2^{(n)}$ can be constructed from 
$U_1\big|_{(0,t_n)}$ 
even if 
one does not want to introduce the splitting proposed above. The second 
equation in \eqref{eq:the_scheme} can indeed by used at any time, as long as 
we 
have the values of $U_1$.
\end{remark}

%%%%%%%%%%%%%%%%%%%%%%%%%%%%%%%%%%%%%%%%%%%%%%%%%%%%%%%%%%%%%%%%%%%%%%%%%%%%%%%
%%%%%%%%%%%%%%%%%%%%%%%%%%%%%%%%%%%%%%%%%%%%%%%%%%%%%%%%%%%%%%%%%%%%%%%%%%%%%%%
%%%%%%%%%%%%%%%%%%%%%%%%%%%%%%%%%%%%%%%%%%%%%%%%%%%%%%%%%%%%%%%%%%%%%%%%%%%%%%%
%%
%%
%%	T H E   R O L E   O F    E A C H    C O M P O N E N T
%%
%%
%%%%%%%%%%%%%%%%%%%%%%%%%%%%%%%%%%%%%%%%%%%%%%%%%%%%%%%%%%%%%%%%%%%%%%%%%%%%%%%
%%%%%%%%%%%%%%%%%%%%%%%%%%%%%%%%%%%%%%%%%%%%%%%%%%%%%%%%%%%%%%%%%%%%%%%%%%%%%%%
%%%%%%%%%%%%%%%%%%%%%%%%%%%%%%%%%%%%%%%%%%%%%%%%%%%%%%%%%%%%%%%%%%%%%%%%%%%%%%%
\subsection{Temporal discretization with polynomials of higher 
degree}\label{subsec:higher_degree}
%%%%%%%%%%%%%%%%%%%%%%%%%%%%%%%%%%%%%%%%%%%%%%%%%%%%%%%%%%%%%%%%%%%%%%%%%%
%%									%%
%%			HIGHER DEGREE					%%
%%									%%
%%%%%%%%%%%%%%%%%%%%%%%%%%%%%%%%%%%%%%%%%%%%%%%%%%%%%%%%%%%%%%%%%%%%%%%%%%
The results in this section can be generalized to polynomials of
arbitrary degree with respect to time. We denote by $S_{k,q+1}$ the
space of continuous functions that are piecewise polynomials of degree
at most $q+1$, with respect to the partition $ \cT_k$, and by
$Q_{k,q}$ the space of discontinuous functions which are piecewise
polynomials of degree at most $q$, for the same partition. We adopt
the same convention and notation as before and define the
finite-dimensional subspaces $\cY_{h,k,q} := Q_{k,q} \otimes V_h $,
and $ \cX_{h,k,q+1} := S_{k,q+1} \otimes V_h$, for some
finite-dimensional subspace $V_h \subset V$.

The discretized problem can be written in variational form as
\begin{equation} \label{eq:discrete_form_q}
U \in {\cY_{h,k,q} \times V_h}\colon \sB_{}^*(U,X) = \sF_{}(X), 
\quad \forall X \in
{\cX_{h,k,q+1}}.
\end{equation}
Results of existence and uniqueness follow from a minor modification of the 
argument used in the case $q=0$, that is, by modifying the norm on the 
space $\cX_{h,k,q+1}$ as follows:
\begin{equation}  \label{eq:discretenorms}
\begin{aligned}
\norm{\cX_{k,q+1}}{X}^2&:= 
\sum_{i=0}^{N-1}{\int_{I_i}{\Big(\norm{V^*}{\dot{X}}^2 
+
\norm{V}{\Pi^{(q)}_i X}^2 \Big)\,\dd s}} + \norm{H}{X(0)}^2,\\
\aabs{\cX_{k,q+1}}{X}^2&:= 
\sum_{i=0}^{N-1}{\int_{I_i}{\Big(\norm{\dot{H}^{-1}}{\dot{X}}^2 
+
\norm{\dot{H}^1}{\Pi^{(q)}_i X}^2 \Big)\,\dd s}} + \norm{H}{X(0)}^2,
\end{aligned}
\end{equation}
where now $\Pi^{(q)}$ is locally defined on each $I_i$ as the orthogonal 
$L^2$-projection onto the space of polynomials of degree at most $q$. In 
particular, the splitting introduced in \S~\ref{subs:decomposition} still 
holds.

%%%%%%%%%%%%%%%%%%%%%%%%%%%%%%%%%%%%%%%%%%%%%%%%%%%%%%%%%%%%%%%%%%%%%%%%%%%%%%%
%%%%%%%%%%%%%%%%%%%%%%%%%%%%%%%%%%%%%%%%%%%%%%%%%%%%%%%%%%%%%%%%%%%%%%%%%%%%%%%
%%%%%%%%%%%%%%%%%%%%%%%%%%%%%%%%%%%%%%%%%%%%%%%%%%%%%%%%%%%%%%%%%%%%%%%%%%%%%%%
%%
%%
%%	T H E   R O L E   O F    E A C H    C O M P O N E N T
%%
%%
%%%%%%%%%%%%%%%%%%%%%%%%%%%%%%%%%%%%%%%%%%%%%%%%%%%%%%%%%%%%%%%%%%%%%%%%%%%%%%%
%%%%%%%%%%%%%%%%%%%%%%%%%%%%%%%%%%%%%%%%%%%%%%%%%%%%%%%%%%%%%%%%%%%%%%%%%%%%%%%
%%%%%%%%%%%%%%%%%%%%%%%%%%%%%%%%%%%%%%%%%%%%%%%%%%%%%%%%%%%%%%%%%%%%%%%%%%%%%%% 
\subsection{The roles of $U_1$ and $U_2$}
In this section we state a result that relates the two 
components of 
$U$ by means of a discretization based 
on the first space-time formulation. We start by considering the 
original problem \eqref{eq:HeatStrong}.
The first space-time formulation \eqref{eq:first_space_time} leads to the 
following discretization:
\begin{equation*}%\label{eq:primal_equiv}
\begin{aligned}
W\in \cX_{h,k,q+1}:\sB(W,Y) = \sF(Y), \quad
\forall Y \in \cY_{h,k,q} \times V_h,  
\end{aligned}
\end{equation*}
while the weak space-time formulation is given in \eqref{eq:discrete_form_q}.
The next theorem states that the discrete solutions to the 
first and to the weak formulations of \eqref{eq:HeatStrong} differ only up 
to a term proportional to the interpolation error of the right-hand side. Since 
this result is not central in this paper, we omit the proof.

\begin{theorem}\label{thm:U_V_non_homo}
If $f^{(\gamma)} \in L^2((0,T);V)$ 
for some $\gamma \in 
\mathbf{N}$, then
\begin{align*}
\norm{L^2((0,T);V)}{U_1-\Pi^{(q)}W} + \norm{H}{ U_2^{(N)} - W(t_N) } 
\leq C k^{\theta+1} \norm{L^2((0,T);V)}{f^{(\theta)}},
\end{align*}
where $\theta:={\min\{q+1,\gamma\}}$.
\end{theorem}

%%%%%%%%%%%%%%%%%%%%%%%%%%%%%%%%%%%%%%%%%%%%%%%%%%%%%%%%%%%%%%%%%%%%%%%%%%
%%%%%%%%%%%%%%%%%%%%%%%%%%%%%%%%%%%%%%%%%%%%%%%%%%%%%%%%%%%%%%%%%%%%%%%%%%
%%%%%%%%%%%%%%%%%%%%%%%%%%%%%%%%%%%%%%%%%%%%%%%%%%%%%%%%%%%%%%%%%%%%%%%%%%
%%
%%
%%		A    P R I O R I    E S T I M A T E S
%%
%%
%%%%%%%%%%%%%%%%%%%%%%%%%%%%%%%%%%%%%%%%%%%%%%%%%%%%%%%%%%%%%%%%%%%%%%%%%%
%%%%%%%%%%%%%%%%%%%%%%%%%%%%%%%%%%%%%%%%%%%%%%%%%%%%%%%%%%%%%%%%%%%%%%%%%%
%%%%%%%%%%%%%%%%%%%%%%%%%%%%%%%%%%%%%%%%%%%%%%%%%%%%%%%%%%%%%%%%%%%%%%%%%%

\section{A priori error estimates}\label{sec:error}
In order to obtain error estimates for our scheme, we first rely on
the quasi-optimality theory, thus achieving an error estimate
consistent with the natural norm of the solution in
\eqref{eq:norm_sol_discrete}. However, numerical experiments (see
Figures~\ref{fig:superconvergence} and \ref{fig:superconvergence_2d})
and Theorem~\ref{thm:second_order_convergence} suggest that the second
component of the solution converges faster, with a rate proportional
to $k^2$. This is consistent with the fact that our method is a
modification of the standard Crank--Nicolson method. By means of a
duality argument we give a rigorous proof of this fact in
Theorem~\ref{thm:superconvergence} in Section~\ref{sec:semidiscrete}.

\subsection{Quasi-optimality}\label{subs:quasi_optimality}
We consider the subspaces 
$\cY_{h,k} \times V_h \subset \cY_H$ and $\cX_{h,k}\subset \cX$ previously 
introduced, endowed with the norms $ \aabs{\cY_H}{\cdot}$ and $ 
\aabs{\cX_k}{\cdot}$, respectively. The following result of 
quasi-optimality holds:

%%%%%%%%%%%%%%%%%%%%%%%%%%%%%%%%%%%%%%%%%%%%%%%%%%%%%%%%%%%%%%%%%%%%%%%%%%%%%%%%
%%%%%%%%%%%%%%%%%%%%%%%%%%%%%%%%%%%%%%%%%%%%%%%%%%%%%%%%%%%%%%%%%%%%%%%%%%%%%%%%
%%%%%%%%%%%%%%%%%%%%%%%%%%%%%%%%%%%%%%%%%%%%%%%%%%%%%%%%%%%%%%%%%%%%%%%%%%%%%%%%
%%
%%
%%	 Q U A S I    O P T I M A L I T Y    T H E O R E M 
%%
%%
%%%%%%%%%%%%%%%%%%%%%%%%%%%%%%%%%%%%%%%%%%%%%%%%%%%%%%%%%%%%%%%%%%%%%%%%%%%%%%%%
%%%%%%%%%%%%%%%%%%%%%%%%%%%%%%%%%%%%%%%%%%%%%%%%%%%%%%%%%%%%%%%%%%%%%%%%%%%%%%%%
%%%%%%%%%%%%%%%%%%%%%%%%%%%%%%%%%%%%%%%%%%%%%%%%%%%%%%%%%%%%%%%%%%%%%%%%%%%%%%%%

\begin{theorem}\label{thm:quasi_opt_unsharp}
If $u$ and $U$ are solutions to \eqref{eq:abstract_problem_param_t} 
and \eqref{eq:discrete_problem}, respectively, the error $u-U$ satisfies the 
following bound:
\begin{equation}\label{eq:quasi_optimality_estimate}
\begin{aligned}
&\Am \norm{L^2((0,t_n);V)}{u_1-U_1}^2 + 
\norm{H}{u_2(t_{n})-U_2^{(n)}}^2 \\
&\quad \leq 
\max\{1,c_{\rm S}\}^2 \Big( \AM \norm{L^2((0,t_n);V)}{u_1-Y_1}^2 + 
\norm{H}{u_2(t_{n})-Y_2^{(n)}}^2 
\Big),
\end{aligned}
\end{equation}
for arbitrary $Y_1 \in \cY_{h,k}$ and $Y_2^{(n)} \in V_h$ and for any $n$.
In particular, it follows that
\begin{equation}\label{eq:quasi_optimality_estimate2}
 \begin{aligned}
 & \Am^{}\norm{L^2((0,T);V)}{u_1-U_1}^2 + 
\max_{i=1,\ldots,N}\norm{H}{u_2(t_{i})-U^{(i)}_2}^2 \\
&\quad \leq
\max\{1,c_{\rm S}\}^2\Big( \AM^{}\norm{L^2((0,T);V)}{u_1-Y_1}^2 + 
\max_{i=1,\ldots,N}\norm{H}{u_2(t_{i})-Y_2^{(i)}}^2 \Big).
\end{aligned}
\end{equation}
\end{theorem}
\begin{proof}
We consider the problem on $(0,t_n)$ with arbitrary $t_n$ and omit
$t_n$ in the notation for the spaces and bilinear form. We denote by
$R \colon \cY_H \mapsto \cY_{h,k}\times V_h$ the Ritz projection,
defined as $R u = U$, that is,
\begin{align}\label{eq:Ritz_projection}
\sB^*(R \phi, X) = \sB^*(\phi, X), \quad \forall X\in \cX_{h,k}.
\end{align}
Since $R$ is idempotent and $\cY_H$ is a Hilbert space, we have $\norm{\sL(\cY_H)}{I-R} = 
\norm{\sL(\cY_H)}{R} $ (see \cite{XuZik}), so that, for any $Y\in
\cY_{h,k}\times V_h$, 
\begin{align*} 
\aabs{\cY_H}{u-U} &=\aabs{\cY_H}{(I-R)u} =\aabs{\cY_H}{(I-R)(u-Y)} 
\leq \norm{\sL(\cY_H)}{R}  \aabs{\cY_H}{u-Y}.
\end{align*}
Here, we have 
\begin{align*}
\norm{\sL(\cY_H)}{R} &=
\sup_{\phi \in \cY_H}\frac{ \aabs{\cY_H}{R \phi} }{ 
\aabs{\cY_H}{\phi}} \leq 
\frac{1}{c_B} \sup_{\phi \in \cY_H} \sup_{X\in
\cX_{h,k}} \frac{\sB^*(R \phi,X)}{\aabs{\cY_H}{\phi}\aabs{\cX_{k}}{X} 
} \\
&=\frac{1}{c_B} \sup_{\phi \in \cY_H} \sup_{X\in
\cX_{h,k}} \frac{\sB^*(\phi,X)}{\aabs{\cY_H}{\phi}\aabs{\cX_{k}}{X} }
\leq 
\frac{C_B}{c_B} \sup_{\phi \in \cY_H} \sup_{X\in 
\cX_{h,k}} \frac{  \aabs{\cY_H}{\phi} 
\aabs{\cX}{X}}{\aabs{\cY_H}{\phi} \aabs{\cX_{k}}{X} }, 
\end{align*}
where we first used the discrete counterpart of \eqref{eq:infsup} with respect 
to $\aabs{\cY_H}{\cdot} $ and $\aabs{\cX_k}{\cdot} $,  then 
\eqref{eq:Ritz_projection}, and  \eqref{eq:bdd}.  
Finally, by means of \eqref{eq:equivalence_norm_cX} we obtain that
\begin{align*}
\frac{C_B}{c_B} \sup_{\phi \in \cY_H} \sup_{X\in
\cX_{h,k}} \frac{  \aabs{\cY_H}{\phi} 
\aabs{\cX}{X}}{\aabs{\cY_H}{\phi} \aabs{\cX_{k}}{X} } \leq 
\frac{C_B}{c_B} \max\{1,c_{\rm S}\} =  \max\{1,c_{\rm S}\},
\end{align*}
since $C_B = c_B = 1$. Since $Y \in {\cY_{h,k}}\times V_h$ is
arbitrary, \eqref{eq:quasi_optimality_estimate} follows by using the
equivalence between the norms $\norm{L^2((0,t_n);V)}{\cdot}$ and
$\norm{L^2((0,t_n);\dot{H}^1)}{\cdot}$. Since $t_n$ is arbitrary, the
second bound \eqref{eq:quasi_optimality_estimate2} follows as well.
\end{proof}

\subsection{Convergence}
We first show convergence of the method under 
minimal assumptions, namely a right-hand side $\sF \in 
\cX^*$ and no further regularity.
\begin{theorem}\label{thm:convergence}
Let $u$ and $U$ be solutions to \eqref{eq:abstract_problem_param_t} and 
\eqref{eq:discrete_problem}, respectively.
If we assume the validity of \eqref{eq:CFL}, and if $\sF \in 
\cX^*$, then $\norm{\cY_H}{u - U} \rightarrow 0 $ as $k,h \rightarrow 0$.
\end{theorem}

\begin{proof}
From the quasi-optimality theorem we have 
\begin{align*}
\norm{\cY_H}{u-U} \leq C \norm{\cY_H}{u-Y}, 
\quad \mbox{for any } Y \in \cY_{h,k}\times V_h,
\end{align*}
where $C$ depends on $\Am$, $\AM$, and $c_{\rm S}$, hence independent
of $h$ and $k$ due to \eqref{eq:CFL}.  We choose $\sV$ to be a space
of sufficiently smooth functions, dense in $\cY_H$, for example
$\sV:=H^1((0,T);V)\times V$.  For arbitrary $\epsilon$, we choose
$v_{\epsilon} \in \sV$ such that, by density, 
\begin{align*}
\norm{\cY_H}{u - v_{\epsilon}} \leq \epsilon/2.
\end{align*}
We then choose $h = h(\epsilon)$ and $k = k(\epsilon)$ such that 
$\tilde{v}_{\epsilon} \in \cY_{h,k} \times V_h$, which denotes the interpolant 
of 
$v_{\epsilon}$, satisfies
\begin{align*}
 \norm{\cY_H}{v_{\epsilon} -\tilde{ v}_{\epsilon}} 
\leq C(h + k)\norm{\sV}{v_{\epsilon}} \leq  \epsilon/2.
\end{align*}
We conclude
\begin{align*}
\norm{\cY_H}{u - U}  \leq C \norm{\cY_H}{u - \tilde{v}_\epsilon} \leq C\Big(
\norm{\cY_H}{u - v_\epsilon} + 
\norm{\cY_H}{v_\epsilon - \tilde{v}_\epsilon} \Big) \leq C \epsilon.
\end{align*}
Since $\epsilon$ is arbitrary, the claim follows.
\end{proof}

\subsection{Convergence of first order in time}

In order to prove the next results we assume that the spatial
discretization is done by using a polynomial space of sufficiently
high degree, so that all the quantities we use make sense and are not
trivial.  This choice is not strictly necessary but it is motivated by
the fact that condition \eqref{eq:CFL} becomes $k \lesssim h^2$ in the
case, for example, of spatial discretization with Lagrange
elements. Thus, in order to have consistency between the spatial and
the temporal rate of convergence, we need to have order $2$ in the
spatial $H^1$-norm in the following theorem (polynomials of degree
$p=2$), and similarly order $4$ in the one after.

We make once again use of the spaces $\dot{H}^{\beta}$ as in 
\S~\ref{subsec:spatial_reg}. The right-hand side of the expression in 
\eqref{eq:quasi_optimality_estimate2} can be further estimated by 
means of standard interpolation estimates, thus we obtain the following 
theorem:
%%%%%%%%%%%%%%%%%%%%%%%%%%%%%%%%%%%%%%%%%%%%%%%%%%%%%%%%%%%%%%%%%%%%%%%%%%%%%%%%
%%%%%%%%%%%%%%%%%%%%%%%%%%%%%%%%%%%%%%%%%%%%%%%%%%%%%%%%%%%%%%%%%%%%%%%%%%%%%%%%
%%%%%%%%%%%%%%%%%%%%%%%%%%%%%%%%%%%%%%%%%%%%%%%%%%%%%%%%%%%%%%%%%%%%%%%%%%%%%%%%
%%
%%
%%	 Q U A S I    O P T I M A L I T Y    T H E O R E M    E X P L
%%
%%
%%%%%%%%%%%%%%%%%%%%%%%%%%%%%%%%%%%%%%%%%%%%%%%%%%%%%%%%%%%%%%%%%%%%%%%%%%%%%%%%
%%%%%%%%%%%%%%%%%%%%%%%%%%%%%%%%%%%%%%%%%%%%%%%%%%%%%%%%%%%%%%%%%%%%%%%%%%%%%%%%
%%%%%%%%%%%%%%%%%%%%%%%%%%%%%%%%%%%%%%%%%%%%%%%%%%%%%%%%%%%%%%%%%%%%%%%%%%%%%%%%
\begin{theorem}\label{thm:error_quasi_optimal_expl}
Let $u$ and $U$ be solutions to \eqref{eq:abstract_problem_param_t} and 
\eqref{eq:discrete_problem}, respectively. For sufficiently smooth
data $f$ and $u_0$, and assuming the validity of \eqref{eq:CFL}, we have:
\begin{equation*}%\label{eq:error_sol_discrete_in_data}
 \begin{aligned}
&\norm{L^2((0,T);V)}{u_1 - U_1} + 
\max_{i=1,\ldots,N}\norm{H}{u_2(t_i) - U^{(i)}_2} \\
&\qquad\leq C( k + h^2) \Big( \norm{L^2((0,T);\dot{H}^1)}{f} + 
\norm{\dot{H}^2}{u_0} \Big).
\end{aligned}
\end{equation*}
\end{theorem}
\begin{proof}
Quasi-optimality \eqref{eq:quasi_optimality_estimate2} and interpolation error 
estimates give us that 
\begin{equation*}%\label{eq:error_sol_discrete}
 \begin{aligned}
&\norm{L^2((0,T);V)}{u_1 - U_1} + 
\max_{i=1,\ldots,N}\norm{H}{u_2(t_i) - U^{(i)}_2} \\
&\quad\leq C \Big( k 
\norm{L^2((0,T);\dot{H}^{1})}{\dot{u}} + h^2\big(\norm{L^2((0,T);\dot{H}^3)}{u} 
+ 
 \max_{i=1,\ldots,N}\norm{\dot{H}^2}{u_2(t_i)} \big) \Big),
\end{aligned}
\end{equation*}
for $u$ sufficiently smooth, with $C$ depending on $c_{\rm S}$, $\Am$,
and $\AM$. In particular, if the right-hand side is of the form
defined in \eqref{eq:regular_rhs}, we can rely on
Theorem~\ref{maintheorem2_beta} with $\beta=2$ to prove the claim.
\end{proof}

%%%%%%%%%%%%%%%%%%%%%%%%%%%%%%%%%%%%%%%%%%%%%%%%%%%%%%%%%%%%%%%%%%%%%%%%%
%%%%%%%%%%%%%%%%%%%%%%%%%%%%%%%%%%%%%%%%%%%%%%%%%%%%%%%%%%%%%%%%%%%%%%%%%
%%%%%%%%%%%%%%%%%%%%%%%%%%%%%%%%%%%%%%%%%%%%%%%%%%%%%%%%%%%%%%%%%%%%%%%%%
%%
%%
%% 	S E C O N D   O R D E R   I N   T I M E   q=0
%%
%%
%%%%%%%%%%%%%%%%%%%%%%%%%%%%%%%%%%%%%%%%%%%%%%%%%%%%%%%%%%%%%%%%%%%%%%%%%
%%%%%%%%%%%%%%%%%%%%%%%%%%%%%%%%%%%%%%%%%%%%%%%%%%%%%%%%%%%%%%%%%%%%%%%%%
%%%%%%%%%%%%%%%%%%%%%%%%%%%%%%%%%%%%%%%%%%%%%%%%%%%%%%%%%%%%%%%%%%%%%%%%%
\subsection{Convergence of second order in time}\label{subsec:superconv}
By means of the connection between first and second discrete space-time 
formulation and by using the fact that the first space-time formulation seen as 
a time stepping coincides with the traditional Crank--Nicolson scheme, we can 
obtain the following result:
\begin{theorem}\label{thm:second_order_convergence}
The scheme in \eqref{eq:the_scheme} converges with a rate proportional to $k^2$ 
at the grid points $\{t_i\}_{i=1,\ldots,N}$ for sufficiently smooth
data. 
\end{theorem}
\begin{proof}
We take $W$ as in Theorem~\ref{thm:U_V_non_homo}, and notice that for every 
$t_i$ we have:
\begin{align*}
\norm{H}{U_2^{(i)} - u_2(t_i)} \leq  \norm{H}{U_2^{(i)} - 
W(t_i)} + \norm{H}{u_2(t_i) - W(t_i)}.
\end{align*}
We can bound the first term by $Ck^2$ according to
Theorem~\ref{thm:U_V_non_homo}. 
The primal formulation produces exactly the 
Crank--Nicolson time stepping, so that the second term is also bounded
by $Ck^2$. 
\end{proof}

%%%%%%%%%%%%%%%%%%%%%%%%%%%%%%%%%%%%%%%%%%%%%%%%%%%%%%%%%%%%%%%%%%%%%%%%
%%%%%%%%%%%%%%%%%%%%%%%%%%%%%%%%%%%%%%%%%%%%%%%%%%%%%%%%%%%%%%%%%%%%%%%%
%%%%%%%%%%%%%%%%%%%%%%%%%%%%%%%%%%%%%%%%%%%%%%%%%%%%%%%%%%%%%%%%%%%%%%%%
%%
%%
%%	s e m i d i s c r e t e    i n    t i m e
%%
%%
%%%%%%%%%%%%%%%%%%%%%%%%%%%%%%%%%%%%%%%%%%%%%%%%%%%%%%%%%%%%%%%%%%%%%%%%
%%%%%%%%%%%%%%%%%%%%%%%%%%%%%%%%%%%%%%%%%%%%%%%%%%%%%%%%%%%%%%%%%%%%%%%%
%%%%%%%%%%%%%%%%%%%%%%%%%%%%%%%%%%%%%%%%%%%%%%%%%%%%%%%%%%%%%%%%%%%%%%%%

\section{Temporal semidiscretization}\label{sec:semidiscrete}
We provide in Theorem~\ref{thm:superconvergence} a direct proof of the
result of Theorem~\ref{thm:second_order_convergence}, that does not
rely on a comparison with the Crank--Nicolson method and that extends
to arbitrary degree.  Following \cite[Theorem~12.3]{Thomee} we present
only the temporally semidiscrete part of the error, since our main
focus is the time discretization.  The proof is based on a duality
argument but first we need to develop a substitute for the
quasi-optimality theory in the semidiscrete case.

\subsection{Existence and uniqueness}
We introduce the following notation for 
the temporally semidiscrete spaces:
\begin{align*}
&\cY_{k,q} := \{ Y\in \cY : Y\big|_{I_i} \in \bP^{q}[t]\otimes \dot{H}^1 
\}, 
&&\cX_{k,q+1} := \{ X\in \cX : X\big|_{I_i} \in \bP^{q+1}[t]\otimes 
\dot{H}^1 \},
\end{align*} 
and we endow $\cX_{k,q+1}$ with the norm $\aabs{\cX_{k,q+1}}{\cdot}$ which we 
introduced in \eqref{eq:discretenorms}.
The semidiscrete problem reads:
\begin{equation} \label{eq:semidiscrete_form_q}
\hat{U} \in {\cY_{k,q} \times H}\colon \sB^*(\hat{U},X) = \sF(X), 
\quad \forall X \in
{\cX_{k,q+1}}.
\end{equation}
In particular, we can split the scheme as in \eqref{eq:the_scheme} in 
order to produce pointwise values of $\hat{U}^{(i)}_2$ at each $t_i$.

Our main concern is to avoid the use of \eqref{eq:equivalence_norm_cX}, 
because $c_{\rm S}$ would not be finite in the semidiscrete 
case. It turns out that a consistent theory 
of existence 
and uniqueness based on the  Banach--Ne{\v{c}}as--Babu{\v{s}}ka can be derived 
even in this case, although more regularity on $f$ must be assumed.
%%%%%%%%%%%%%%%%%%%%%%%%%%%%%%%%%%%%%%%%%%%%%%%%%%%%%%%%%%%%%%%%%%%%%%%%%%%
%%%%%%%%%%%%%%%%%%%%%%%%%%%%%%%%%%%%%%%%%%%%%%%%%%%%%%%%%%%%%%%%%%%%%%%%%%%
%%%%%%%%%%%%%%%%%%%%%%%%%%%%%%%%%%%%%%%%%%%%%%%%%%%%%%%%%%%%%%%%%%%%%%%%%%%
%%
%%
%%		L E M M A  / B O U N D E D N E S S
%%
%%
%%%%%%%%%%%%%%%%%%%%%%%%%%%%%%%%%%%%%%%%%%%%%%%%%%%%%%%%%%%%%%%%%%%%%%%%%%%
%%%%%%%%%%%%%%%%%%%%%%%%%%%%%%%%%%%%%%%%%%%%%%%%%%%%%%%%%%%%%%%%%%%%%%%%%%%
%%%%%%%%%%%%%%%%%%%%%%%%%%%%%%%%%%%%%%%%%%%%%%%%%%%%%%%%%%%%%%%%%%%%%%%%%%%
We start by presenting a semidiscrete version of Lemma~\ref{lemma:equiv_norm}:

\begin{lemma}\label{lemma:equiv_norm_sd}
The norm $\nnnorm{\cX_{k,q+1}}{\cdot}$, defined by
\begin{align*}
\nnorm{\cX_{k,q+1}}{X} := \norm{H}{X(t)}^2 + 
\sum_{i=1}^{N-1}\int_{I_i} \norm{H}{A^{\frac12}\Pi^{(q)} X(s) 
-A^{-\frac12}\dot{X}(s)}^2\,\dd s 
\end{align*}
is equal on $ \cX_{k,q+1} $ to the norm $\aabs{\cX_{k,q+1}}{\cdot}$.
\end{lemma}
\begin{proof}
Similarly to the proof of Lemma~\ref{lemma:equiv_norm}, we have 
\begin{align*}
\nnorm{\cX_{k,q+1}}{X}
&= \norm{H}{X(t)}^2 + \sum_{i=1}^{N-1}\int_{I_i} 
\Big(\norm{H}{A^{\frac12} \Pi_i^{(q)}X}^2 
+\norm{H}{A^{-\frac12}\dot{X}}^2 \\
&\qquad \qquad \qquad \qquad \qquad -  2\tdual{\Pi_i^{(q)}X}{\dot{X}}
\Big)\,\dd s\\
&= \norm{H}{X(t)}^2 + \sum_{i=1}^{N-1}\int_{I_i} 
\Big(\norm{H}{A^{\frac12} \Pi_i^{(q)}X}^2 
+\norm{H}{A^{-\frac12}\dot{X}}^2 \\
&\qquad \qquad \qquad \qquad \qquad - \norm{H}{X(t_{i+1})}^2 +  
\norm{H}{X(t_i)}^2
\Big)\,\dd s\\
&= \norm{H}{X(0)}^2 + \sum_{i=0}^{N-1}
\int_{I_i}\Big(\norm{H}{A^{\frac12} \Pi_i^{(q)}X}^2 
+\norm{H}{A^{-\frac12}\dot{X}}^2 \Big)\,\dd s =
  \aabs{\cX_{k,q+1}}{x}^2, 
\end{align*}
since $\int_{I_i}\tdual{\Pi_i^{(q)}X}{\dot{X}}\,\dd s=
\int_{I_i}\tdual{X}{\dot{X}}\,\dd s$.  This is the desired result.
\end{proof}

%%%%%%%%%%%%%%%%%%%%%%%%%%%%%%%%%%%%%%%%%%%%%%%%%%%%%%%%%%%%%%%%%%%%%%%%%%%%%
%%%%%%%%%%%%%%%%%%%%%%%%%%%%%%%%%%%%%%%%%%%%%%%%%%%%%%%%%%%%%%%%%%%%%%%%%%%%%
%%%%%%%%%%%%%%%%%%%%%%%%%%%%%%%%%%%%%%%%%%%%%%%%%%%%%%%%%%%%%%%%%%%%%%%%%%%%%
%%
%%
%% 	B N B  / T H E O R E M 
%%
%%
%%%%%%%%%%%%%%%%%%%%%%%%%%%%%%%%%%%%%%%%%%%%%%%%%%%%%%%%%%%%%%%%%%%%%%%%%%%%%
%%%%%%%%%%%%%%%%%%%%%%%%%%%%%%%%%%%%%%%%%%%%%%%%%%%%%%%%%%%%%%%%%%%%%%%%%%%%%
%%%%%%%%%%%%%%%%%%%%%%%%%%%%%%%%%%%%%%%%%%%%%%%%%%%%%%%%%%%%%%%%%%%%%%%%%%%%%

\begin{theorem}\label{thm:bnbthm_sd}
The bilinear form \eqref{eq:bilinear_form} satisfies the following:
\begin{align}
&C_B:=\sup_{0\neq Y \in \cY_{k,q} \times H}\sup_{0\neq X \in \cX_{k,q+1}} 
\frac{\sB^*(Y,X)}{\aabs{\cY_H}{Y} \aabs{\cX_{k,q+1}}{X}  } = 1, 
\label{eq:bdd_sd}\\
&c_B:=\inf_{0\neq Y \in \cY_{k,q} \times H} \sup_{0\neq X \in \cX_{k,q+1}} 
\frac{\sB^*(Y,X)}{\aabs{\cY_H}{Y} \aabs{\cX_{k,q+1}}{X} } = 
1.\label{eq:infsup_sd}
\end{align}
\end{theorem}
\begin{proof}
We first notice that, on each $I_i$,
\begin{align*}
\int_{I_i}{ \tdual{Y(s)}{-\dot{X}(s) + A^* X(s)} }\,\dd s = \int_{I_i}{ 
\tdual{Y(s)}{-\dot{X}(s) + A^* \Pi^{(q)}_i X(s)} }\,\dd s,
\end{align*}
so that we can use H\"older's inequality as in the proof of 
Theorem~\ref{thm:bnbthm} 
and obtain \eqref{eq:bdd_sd}. The proof of \eqref{eq:infsup_sd} follows by 
choosing, for $X \in \cX_{k,q+1}$,
\begin{align*}
Y_X = \Big( \Pi^{(q)}X - A^{-1}\dot{X}, X(t) \Big),
\end{align*}
and proceeding in the same way as in the continuous case.
\end{proof}
Since we are in a semidiscrete case, the conditions \eqref{BNB1A} and 
\eqref{BNB2A} are not equivalent, and one should prove also the latter. We 
refrain from doing so and refer to \cite[Proposition 4.2]{Fra}, where a complete 
proof for the case $q=0$ can be found. The case of $q>0$ follows similarly.
In order to have solvability of \eqref{eq:semidiscrete_form_q} it now 
only remains to bound $\sF$ with respect to the norm 
$\aabs{\cX_{k,q+1}}{\cdot}. $

%%%%%%%%%%%%%%%%%%%%%%%%%%%%%%%%%%%%%%%%%%%%%%%%%%%%%%%%%%%%%%%%%%%%%%%%%%%%%%%%
%%%%%%%%%%%%%%%%%%%%%%%%%%%%%%%%%%%%%%%%%%%%%%%%%%%%%%%%%%%%%%%%%%%%%%%%%%%%%%%%
%%%%%%%%%%%%%%%%%%%%%%%%%%%%%%%%%%%%%%%%%%%%%%%%%%%%%%%%%%%%%%%%%%%%%%%%%%%%%%%%
%%
%%
%% 	R H S / S E M I D I S C R E T E 
%%
%%
%%%%%%%%%%%%%%%%%%%%%%%%%%%%%%%%%%%%%%%%%%%%%%%%%%%%%%%%%%%%%%%%%%%%%%%%%%%%%%%%
%%%%%%%%%%%%%%%%%%%%%%%%%%%%%%%%%%%%%%%%%%%%%%%%%%%%%%%%%%%%%%%%%%%%%%%%%%%%%%%%
%%%%%%%%%%%%%%%%%%%%%%%%%%%%%%%%%%%%%%%%%%%%%%%%%%%%%%%%%%%%%%%%%%%%%%%%%%%%%%%%
%%%%%%%%%%%%%%%%%%%%%%%%%%%%%%%%%%%%%%%%%%%%%%%%%%%%%%%%%%%%%%%%%%%%%%%%%%%%%%%%

\begin{lemma}\label{lemma:RHS_semidiscrete}
If $f\in L^2((0,T);\dot{H}^1)$ and $u_0 \in H$, then we have for $X \in 
\cX_{k,q+1}$ the 
following inequality:
\begin{align*}
\Big| \sF(X) \Big| \leq \Big[ \sum_{i=0}^{N-1}\Big(\int_{I_i} 
{\norm{\dot{H}^{-1}}{f(s)}^2 }\,\dd s + 
k_i^2 \int_{I_i}{ \norm{\dot{H}^1}{f(s)}^2\,\dd 
s \Big) + \norm{H}{u_0}^2} \Big]^{\frac12} 
\aabs{\cX_{k,q+1}}{X}. 
\end{align*}
\end{lemma}
\begin{proof}
We use the fact that, for $X \in \cX_{k,q+1}$ and for every subinterval $I_i$, 
we have
\begin{align}\label{eq:interpolation_err_sd}
\norm{L^2(I_i;\dot{H}^{-1})}{X - \Pi^{(q)}_i X}^2 \leq 
\norm{L^2(I_i;\dot{H}^{-1})}{X - \Pi^{(0)}_i X}^2 \leq
k_i^2\norm{L^2(I_i;\dot{H}^{-1})}{\dot{X}}^2 .
\end{align}
By adding and subtracting $\Pi^{(q)} X$, we have
\begin{align*}
\sF(X) &= \sum_{i=0}^{N-1}\Big( \int_{I_i}{ \inner{f(s)}{\Pi^{(q)}_i X(s)} 
}\,\dd s  + \int_{I_i}{ \inner{f(s)}{X(s) - \Pi^{(q)}_i X(s)} }\,\dd s \Big)\\
&\quad + 
\inner{u_0}{X(0)},
\end{align*}
so that 
\begin{align*}
\Big| \sF(X) \Big| 
& \leq 
\sum_{i=0}^{N-1}  \Big(  \norm{L^2(I_i;\dot{H}^{-1})}{f}
 \norm{L^2(I_i;\dot{H}^{1})}{\Pi_i^{(q)}X} 
+  \norm{L^2(I_i;\dot{H}^{1})}{f}
  k_i\norm{L^2(I_i;\dot{H}^{-1})}{\dot{X}} \Big)
\\ & \quad 
+ \norm{H}{u_0}\norm{H}{X(0)}
\\ &
\le 
\Big[ \sum_{i=0}^{N-1} \Big( \norm{L^2(I_i;\dot{H}^{-1})}{f}^2 
+k_i^2 \norm{L^2(I_i;\dot{H}^{1})}{f}^2\Big)
+ \norm{H}{u_0}^2 
\Big]^{\frac12} \aabs{\cX_{k,q+1}}{X},
\end{align*}
which proves the 
claim.
\end{proof}
The previous lemma shows in particular that 
\begin{align*}
\norm{(\cX_{k,q+1}, \aabs{\cX_{k,q+1}}{\cdot})^*}{\sF} \leq
\Big(
\norm{L^2((0,T);\dot{H}^{-1})}{f}^2 +k^2 \norm{L^2((0,T);\dot{H}^1)}{f}^2 + 
\norm{H}{u_0}^2\Big)^{\frac12}, 
\end{align*}
so that the next theorem follows:
\begin{theorem}\label{thm:existence_uniqueness_semi}
If $f\in L^2((0,T);\dot{H}^1)$ and $u_0 \in H$, there exists a unique solution 
$\hat{U} \in \cY_{k,q}\times H$ to the 
semidiscrete problem, and its norm is such that
\begin{align*}
\aabs{\cY_H}{\hat{U}} \leq \Big(
\norm{L^2((0,T);\dot{H}^{-1})}{f}^2 +k^2 
\norm{L^2((0,T);\dot{H}^1)}{f}^2 + 
\norm{H}{u_0}^2\Big)^{\frac12}.
\end{align*}
\end{theorem}

\subsection{A priori error estimate}
In the proof of Theorem~\ref{thm:quasi_opt_unsharp} we relied on the
boundedness of $\sB^*$ with respect to $\aabs{\cX}{\cdot}$ and
$\aabs{\cY_H}{\cdot}$, together with the norm equivalence
\eqref{eq:equivalence_norm_cX} between $\aabs{\cX}{\cdot}$ and
$\aabs{\cX_k}{\cdot}$, to show its boundedness with respect to
$\aabs{\cX_k}{\cdot}$ and $\aabs{\cY_H}{\cdot}$. This does not work
here due to the fact that the constant $c_{\rm S}$, that would appear,
is not finite in the semidiscrete case. We solve this problem by
bounding the bilinear form with respect to
$\aabs{\cX_{k,q+1}}{\cdot} $ and a stronger norm on $\cY$.

%%%%%%%%%%%%%%%%%%%%%%%%%%%%%%%%%%%%%%%%%%%%%%%%%%%%%%%%%%%%%%%%%
%%%%%%%%%%%%%%%%%%%%%%%%%%%%%%%%%%%%%%%%%%%%%%%%%%%%%%%%%%%%%%%%%
%%%%%%%%%%%%%%%%%%%%%%%%%%%%%%%%%%%%%%%%%%%%%%%%%%%%%%%%%%%%%%%%%
%%
%%
%% 	L E M M A   / A L T   B O U N D E D N E S S 
%%
%%
%%%%%%%%%%%%%%%%%%%%%%%%%%%%%%%%%%%%%%%%%%%%%%%%%%%%%%%%%%%%%%%%%
%%%%%%%%%%%%%%%%%%%%%%%%%%%%%%%%%%%%%%%%%%%%%%%%%%%%%%%%%%%%%%%%%
%%%%%%%%%%%%%%%%%%%%%%%%%%%%%%%%%%%%%%%%%%%%%%%%%%%%%%%%%%%%%%%%%

\begin{lemma}\label{lemma:alternative_boundedness}
The following boundedness estimate holds for any $X\in\cX_{k,q+1}$ and $y \in 
L^2((0,t_n);\dot{H}^3) \times H$ such that $y_2=0 $:
\begin{align*}
 |\sB^*(y,X)| \leq C \Big[ \sum_{i=0}^{N-1} \Big( \int_{I_i}{ \norm{ 
\dot{H}^{1} 
}{y}^2}\,\dd s + 
k_i^2\int_{I_i}{ 
\norm{\dot{H}^{3}}{y}^2 }\,\dd s\Big) \Big]^{\frac12}\aabs{\cX_{k,q+1}}{X}.
\end{align*}
\end{lemma}
\begin{proof}
The term we need to modify in order to achieve the 
$\aabs{\cX_{k,q+1}}{\cdot}$-norm, 
is the one not involving the time derivative. For this term we have 
\begin{align*}
\int_{I_i}{ \inner{y}{A^* X} }\,\dd s = \int_{I_i}{ \inner{Ay}{ 
\Pi^{(q)}_i X}}\,\dd s + \int_{I_i}{ 
\inner{Ay}{X - \Pi^{(q)}_i X}}\,\dd s .
\end{align*}
If we now take norms and use \eqref{eq:interpolation_err_sd}, we get
\begin{align*}
\Big|\int_{I_i}{ \inner{y}{A^* X} }\,\dd s \Big| 
\leq 
\norm{ L^2(I_i;\dot{H}^{1}) }{y}\norm{L^2(I_i;\dot{H}^{1} )}{\Pi^{(q)}_i X}
+ k_i
\norm{ L^2(I_i;\dot{H}^{3})}{y}\norm{ L^2(I_i;\dot{H}^{-1}) }{\dot{X}} .
\end{align*}
This proves the claim.
\end{proof}
We can now prove a substitute for a quasi-optimality theorem for the
semidiscrete case.

%%%%%%%%%%%%%%%%%%%%%%%%%%%%%%%%%%%%%%%%%%%%%%%%%%%%%%%%%%%%%%%%%%%%%%%%%%%%%
%%%%%%%%%%%%%%%%%%%%%%%%%%%%%%%%%%%%%%%%%%%%%%%%%%%%%%%%%%%%%%%%%%%%%%%%%%%%%
%%%%%%%%%%%%%%%%%%%%%%%%%%%%%%%%%%%%%%%%%%%%%%%%%%%%%%%%%%%%%%%%%%%%%%%%%%%%%
%%
%%
%% 	Q U A S I - O P T I M A L I T Y 
%%
%%
%%%%%%%%%%%%%%%%%%%%%%%%%%%%%%%%%%%%%%%%%%%%%%%%%%%%%%%%%%%%%%%%%%%%%%%%%%%%%
%%%%%%%%%%%%%%%%%%%%%%%%%%%%%%%%%%%%%%%%%%%%%%%%%%%%%%%%%%%%%%%%%%%%%%%%%%%%%
%%%%%%%%%%%%%%%%%%%%%%%%%%%%%%%%%%%%%%%%%%%%%%%%%%%%%%%%%%%%%%%%%%%%%%%%%%%%%

\begin{theorem}\label{thm:quasi_opt_mod}
If $u$ and $\hat{U}$ are solutions to 
\eqref{eq:secondspacetime} and \eqref{eq:semidiscrete_form_q}, respectively, 
then the 
error 
$u-\hat{U}$ satisfies the following bound:
\begin{equation*}
\aabs{\cY_H}{u- \hat{U} } \leq C \Big[ \sum_{i=0}^{N-1} \Big( \int_{I_i}{ 
\norm{ 
\dot{H}^{1} 
}{u_1 - Y_1}^2}\,\dd s + 
k_i^2\int_{I_i}{ 
\norm{\dot{H}^{3}}{u_1-Y_1}^2 }\,\dd s\Big) \Big]^{\frac12},
\end{equation*}
for any $Y_1 \in \cY_{k,q,3} := \{ Y\in \cY \colon 
Y\big|_{I_i} \in 
\bP^q[t]\otimes \dot{H}^{3} \}$.
\end{theorem}
\begin{proof}
We notice that we have the orthogonality
\begin{equation*}
\sB^*(u-\hat{U},X)=0, \quad \forall X \in {\cX_{k,q+1}},
\end{equation*}
so that, for any $Y \in {\cY_{k,q}} \times H$,
\begin{align*}
\aabs{\cY_H}{u- \hat{U} } &\leq \aabs{\cY_H}{u-Y} + 
\aabs{\cY_H}{\hat{U}-Y}\\
&\leq \aabs{\cY_H}{u-Y} + \sup_{X\in
\cX_{k,q+1}}\frac{\sB^*(\hat{U}-Y,X)}{\aabs{\cX_{k,q+1}}{X}}\\
&= \aabs{\cY_H}{u-Y} +\sup_{X\in 
\cX_{k,q+1}}\frac{\sB^*(u-Y,X) +
\sB^*(\hat{U}-u,X) }{\aabs{\cX_{k,q+1}}{X}}\\
&= \aabs{\cY_H}{u-Y} + \sup_{X\in
\cX_{k,q+1}}\frac{\sB^*(u-Y,X)}{\aabs{\cX_{k,q+1}}{X}},
\end{align*}
where the first inequality comes from \eqref{eq:infsup_sd}, while the
last equality comes from orthogonality.  If we choose $Y$ such that
its second component is equal to $u_2$, which is possible in the
semidiscrete case, then we have $ Y_2 - u_2=0$, so that
Lemma~\ref{lemma:alternative_boundedness} applies, giving:
\begin{align*}
&\aabs{\cY_H}{u- \hat{U} } \leq C \Big[ \sum_{i=0}^{N-1} \Big( 
\int_{I_i}{ \norm{ 
\dot{H}^{1} 
}{u_1 - Y_1}^2}\,\dd s + 
k_i^2\int_{I_i}{ 
\norm{\dot{H}^{3}}{u_1-Y_1}^2 }\,\dd s\Big) \Big]^{\frac12},
\end{align*}
for any arbitrary $Y_1 \in \cY_{k,q,3} := \{ Y\in \cY \colon 
Y\big|_{I_i} \in 
\bP^q[t]\otimes \dot{H}^{3} \}$.
\end{proof}

Note that in this proof we cannot use $\norm{}{I-R} = \norm{}{R}$, as in the 
proof of Theorem~\ref{thm:quasi_opt_unsharp}, because we use different norms 
on $U$ and $u$ in $U=Ru$.  

\begin{remark}
It is worth noticing that everything said so far still holds when we shift 
spatial regularity and work with a solution $u \in 
L^2((0,T);\dot{H}^{\beta+1})$; it is easy to see that this leads 
to the following modified inequality:
\begin{equation*}%\label{eq:modified_semidiscrete}
\begin{aligned}
&\norm{ L^2((0,T);\dot{H}^{\beta+1} ) \times \dot{H}^{\beta}  }{u- \hat{U} } \\
&\qquad \qquad \leq C \Big[ \sum_{i=0}^{N-1} \Big( \int_{I_i}{ \norm{ 
\dot{H}^{\beta+1}
}{u_1 - Y_1}^2}\,\dd s + 
k_i^2\int_{I_i}{ 
\norm{\dot{H}^{\beta +3}}{u_1-Y_1}^2 }\,\dd s\Big) \Big]^{\frac12},
\end{aligned}
\end{equation*}
for any $ Y_1$ in the space $ \cY_{k,q,\beta+3} := \{ Y\in \cY \colon 
Y\big|_{I_i} \in 
\bP^q[t]\otimes \dot{H}^{\beta+3} \}$,
\end{remark}

\subsection{Convergence of order $q+1$}
Now that we have an abstract error estimate for the semidiscrete case,
we can derive an analogue to Theorem~\ref{thm:error_quasi_optimal_expl}.

\begin{theorem}\label{thm:error_quasi_optimal_expl_sd}
  For sufficiently smooth data, the error in the
  semidiscrete scheme \eqref{eq:semidiscrete_form_q} satisfies the
  following inequality, for $\beta \geq 0$,
\begin{equation*}%\label{eq:conv_q_sd}
 \begin{aligned}
&\norm{L^2((0,T);\dot{H}^{\beta +1})}{u_1 - \hat{U}_1} + 
\max_{i=1,\ldots,N}\norm{\dot{H}^{\beta}}{u_2(t_i) - \hat{U}^{(i)}_2} \\
&\qquad  \leq C  \Big[ \sum_{i=0}^{N-1} k_i^{2(q+1)} \Big(
\norm{L^2(I_i;\dot{H}^{\beta+ 3})}{u_1^{(q)}}^2 + 
\norm{L^2(I_i;\dot{H}^{\beta+1})}{u_1^{(q+1)}}^2 \Big) 
\Big]^{\frac12}.
\end{aligned}
\end{equation*}
\end{theorem}

\subsection{Pointwise superconvergence of order $2(q+1)$}

We can now give a rigorous proof of Theorem~\ref{thm:second_order_convergence} that 
does not rely on the explicit form of the scheme obtained by discretizing with 
the first space-time formulation. The advantage of an explicit proof is that it 
holds for any arbitrary $q$, while Theorem~\ref{thm:second_order_convergence} relies on 
the fact that the particular time stepping obtained for the first space-time 
formulation of \eqref{eq:HeatStrong} is the Crank--Nicolson method.

%%%%%%%%%%%%%%%%%%%%%%%%%%%%%%%%%%%%%%%%%%%%%%%%%%%%%%%%%%%%%%%%%%%%%%%
%%%%%%%%%%%%%%%%%%%%%%%%%%%%%%%%%%%%%%%%%%%%%%%%%%%%%%%%%%%%%%%%%%%%%%%
%%%%%%%%%%%%%%%%%%%%%%%%%%%%%%%%%%%%%%%%%%%%%%%%%%%%%%%%%%%%%%%%%%%%%%%
%%
%%
%% 	S U P E R C O N V E R G E N C E   / A R B I T R A R Y  q
%%
%%
%%%%%%%%%%%%%%%%%%%%%%%%%%%%%%%%%%%%%%%%%%%%%%%%%%%%%%%%%%%%%%%%%%%%%%%
%%%%%%%%%%%%%%%%%%%%%%%%%%%%%%%%%%%%%%%%%%%%%%%%%%%%%%%%%%%%%%%%%%%%%%%
%%%%%%%%%%%%%%%%%%%%%%%%%%%%%%%%%%%%%%%%%%%%%%%%%%%%%%%%%%%%%%%%%%%%%%%

\begin{theorem}\label{thm:superconvergence}
  For sufficiently smooth data, the numerical solution obtained by
  splitting \eqref{eq:semidiscrete_form_q} is superconvergent at the
  grid points, that is,
  \begin{equation}\label{eq:localises_superconverg}
  \begin{aligned}
&\max_{n=1,\ldots,N}\norm{H}{u_2(t_n) - U_2^{(n)}} \\
&\qquad  \leq C  k^{q+1} 
\Big[ \sum_{i=0}^{N-1} k_i^{2(q+1)} \Big(
\norm{L^2(I_i;\dot{H}^{2q+5})}{u_1^{(q)}}^2 
+ \norm{L^2(I_i;\dot{H}^{2q+3})}{u_1^{(q+1)}}^2  \Big)
\Big]^{\frac12},
\end{aligned}
  \end{equation}
or, in terms of the data,
  \begin{equation}\label{eq:globalised_superconverg}
  \begin{aligned}
  &\max_{n=1,\ldots,N}\norm{H}{u_2(t_n) - U_2^{(n)}} \\
  &\qquad \leq C  
k^{2(q+1)} \Big( 
\norm{L^2((0,T);\dot{H}^{2q+3})}{f^{(q)}} + 
\norm{\dot{H}^{2q+4}}{u_{q,0}} \Big),
\end{aligned}
  \end{equation}
where $u_{q,0}$ is defined as:
\begin{equation*}
u_{q,0} := \sum_{k=0}^{q-1}{(-A)^k f^{(q-1-k)}(0)} + (-A)^{q}u_0. 
\end{equation*}
\end{theorem}
\begin{proof}
  We consider the problem on $(0,t_n)$ with arbitrary $t_n$ and omit
  $t_n$ in the notation for the spaces and bilinear form.  The
  following orthogonality relation is satisfied, for $e = u- \hat{U}$:
\begin{equation}\label{eq:orthogonality_error}
\sB^*(e,X) = 0, \quad \forall X \in \cX_{k,q+1}.
\end{equation}
We now consider the adjoint problem given by
\begin{equation*}%\label{eq:HeatStrong_Adjoint}
\begin{aligned}
-\dot{z}(s) + Az(s) = 0, \quad &\mbox{in }V^*, \, s \in (0,t_n),\\
z(t_n) = \phi,\quad  &\mbox{in }H,
\end{aligned}
\end{equation*}
where $\phi$ is an arbitrary element of $H$.
The first space-time formulation of this problem is given in the continuous 
case by
\begin{align}\label{eq:cont_primal}
&z \in \cX \colon \sB^*(y,z) = \inner{y_2}{\phi}, \quad \forall y = 
(y_1,y_2) \in 
\cY_H.  
\end{align}
In particular, if we choose $ y = (0,e_2)$ in \eqref{eq:cont_primal} and use 
the orthogonality relation \eqref{eq:orthogonality_error}, we have that for 
any $X \in 
\cX_{k,q+1}$:
\begin{align*}
\inner{e_2}{\phi} = \sB^*(e,z) = \sB^*(e,z-X).
\end{align*}
If we assume that we have sufficient smoothness for the next quantities to make 
sense, we have:
\begin{align*}
\abs{ \inner{e_2}{\phi}   } \leq \norm{L^2((0,t_n);\dot{H}^{\beta+1}) \times 
\dot{H}^{\beta}}{e} \norm{L^2((0,t_n);\dot{H}^{1-\beta})\cap 
H^1((0,t_n);\dot{H}^{-1-\beta})}{z-X}.
\end{align*}
For the second term we choose $X \in\cX_{k,q+1}$ to be a standard interpolant of $z$:
\begin{align*}
&\norm{L^2((0,t_n);\dot{H}^{1-\beta})\cap 
H^1((0,t_n);\dot{H}^{-1-\beta})}{z-X} \\
&\qquad \leq C k^{q+1}\Big( \norm{L^2((0,t_n);\dot{H}^{1-\beta})}{z^{(q+1)}}  + 
\norm{L^2((0,t_n);\dot{H}^{-1-\beta})}{z^{(q+2)}}\Big)   \\
&\qquad=  Ck^{q+1}\Big( \norm{L^2((0,t_n);\dot{H}^{1-\beta})}{A^{q+1} z}  + 
\norm{L^2((0,t_n);\dot{H}^{-1-\beta})}{A^{q+1} \dot{z}} \Big)  \\
&\qquad = Ck^{q+1}\Big( \norm{L^2((0,t_n);\dot{H}^{1-\beta+(2q+2)})}{ z}  + 
\norm{L^2((0,t_n);\dot{H}^{-1-\beta+(2q+2)})}{ \dot{z}} \Big)\\ 
&\qquad = Ck^{q+1}\Big( \norm{L^2((0,t_n);\dot{H}^{1})}{ z}  + 
\norm{L^2((0,t_n);\dot{H}^{-1})}{ \dot{z}} \Big)
\le Ck^{q+1} \norm{H}{\phi},  
\end{align*}
where we chose $\beta = 2(q+1)$ and used a standard bound for $z$.
Hence,
\begin{align*}
\norm{H}{e_2} \leq C k^{q+1} \norm{L^2((0,t_n);\dot{H}^{2(q+1)+1}) \times 
\dot{H}^{2(q+1)}}{e},
\end{align*}
and \eqref{eq:localises_superconverg} follows 
by Theorem~\ref{thm:error_quasi_optimal_expl_sd} and recalling that $n$ is 
arbitrary.

In order to show \eqref{eq:globalised_superconverg}, we notice that 
\eqref{eq:localises_superconverg} implies the non-localized bound
\begin{align*}
\max_{n=1,\ldots,N}\norm{H}{e_2^{(n)}} \leq C k^{2(q+1)} \Big( 
\norm{L^2((0,T);\dot{H}^{2q+5})}{u_1^{(q)}} + 
\norm{L^2((0,T);\dot{H}^{2q+3})}{u_1^{(q+1)}} \Big).
\end{align*}
The final step is achieved by bounding the norm of the solution in terms of the 
norm of its data. By using the notation $ u_q := u^{(q)}$, and noticing that 
$u_q$ is the solution to the primal formulation of 
\begin{equation*}
\dot{u}_q + Au_q = f^{(q)},\ t\in (0,T); \quad 
u_q(0) = u_{q,0},% := \sum_{k=0}^{q-1}{A^k f^{(q-1-k)}(0)} + A^{q}u_0,
\end{equation*}
we can see that the boundedness of $
\norm{L^2((0,T);\dot{H}^{2q+5})}{u^{(q)}} + 
\norm{L^2((0,T);\dot{H}^{2q+3})}{u^{(q+1)}} $, is equivalent to  
$u_q \in L^2( (0,T);\dot{H}^{2q+5}) \cap  H^1( (0,T);\dot{H}^{2q+3})$.
According to Theorem~\ref{maintheorem2_beta} a sufficient condition for this is 
given by $ f^{(q)} \in L^2((0,T);\dot{H}^{2q+3})$ and $u_{q,0} \in 
\dot{H}^{2q+4}$, which gives
\begin{align*}
\norm{L^2((0,T);\dot{H}^{2q+5})}{u_q}^2 + 
\norm{L^2((0,T);\dot{H}^{2q+3})}{\dot{u}_q}^2 \leq 
\norm{L^2((0,T);\dot{H}^{2q+3})}{f^{(q)}}^2 + 
\norm{\dot{H}^{2q+4}}{u_{q,0}}^2 .
\end{align*}
We thus achieve the final estimate 
\begin{align*}
\max_{n=1,\ldots,N}\norm{H}{e_2^{(n)}} \leq C k^{2(q+1)} \Big( 
\norm{L^2((0,T);\dot{H}^{2q+3})}{f^{(q)}} + 
\norm{\dot{H}^{2q+4}}{u_{q,0}} \Big),
\end{align*}
which completes the proof.
\end{proof}

\begin{remark}
Theorem~\ref{thm:superconvergence} shows a gain of an extra factor 
$k^{q+1}$, which comes from the duality argument and interpolation of degree 
$q+1$ in the $H^1(I_i;\dot{H}^s)$-norm (Aubin--Nitsche trick). A similar 
argument in \cite[Theorem~12.3]{Thomee} for the ${\rm dG}(q)$-method yields only 
a factor $k^{q}$ because the test functions are of degree $q$.
\end{remark}

%%%%%%%%%%%%%%%%%%%%%%%%%%%%%%%%%%%%%%%%%%%%%%%%%%%%%%%%%%%%%%%%%%%%%%%%
%%%%%%%%%%%%%%%%%%%%%%%%%%%%%%%%%%%%%%%%%%%%%%%%%%%%%%%%%%%%%%%%%%%%%%%%
%%%%%%%%%%%%%%%%%%%%%%%%%%%%%%%%%%%%%%%%%%%%%%%%%%%%%%%%%%%%%%%%%%%%%%%%
%%
%%
%%	N U M E R I C A L    E X P E R I M E N T S
%%
%%
%%%%%%%%%%%%%%%%%%%%%%%%%%%%%%%%%%%%%%%%%%%%%%%%%%%%%%%%%%%%%%%%%%%%%%%%
%%%%%%%%%%%%%%%%%%%%%%%%%%%%%%%%%%%%%%%%%%%%%%%%%%%%%%%%%%%%%%%%%%%%%%%%
%%%%%%%%%%%%%%%%%%%%%%%%%%%%%%%%%%%%%%%%%%%%%%%%%%%%%%%%%%%%%%%%%%%%%%%%

\section{Numerical experiments}\label{sec:numerics}
Since our main concern is about the temporal evolution of the problem, we 
restrict the numerical tests to the case of one and two spatial dimensions, 
discretized by means of Lagrangian elements of sufficiently high degree so 
that the dominating term in the error is given by the temporal part. We test 
for two different problems the validity of our \emph{a priori} estimates. In both 
cases we impose the validity of condition \eqref{eq:CFL} by taking $k = h^2$.
\subsection{One-dimensional test}
We test our scheme for the following problem on the space-time domain 
$(0,1)\times(0,1]$:
\begin{equation}\label{eq:numerical_problem}
\begin{aligned}
&\dot{u}(\xi,t) - u''(\xi,t) = 2\,\pi\,\sin(2\,\pi\,\xi)\Big( 
\cos(2\,\pi\,t) + 2\,\pi\,\sin(2\,\pi\,t) \Big),\\
&u(0,t) = u(1,t) = 0,\quad  t\in[0,1], \\
&u(\xi,0) = 0,\quad  \xi \in[0,1], \\
\end{aligned}
\end{equation}
which has the solution $u(\xi,t) = \sin(2\,\pi\,\xi)\sin(\pi\,t)$. 

In Figure~\ref{fig:error_decay} we report a log-log graph showing the decay of 
the 
error normalized by the norm of the right-hand side, for the numerical solution 
of Problem \eqref{eq:numerical_problem}. In Figure~\ref{fig:superconvergence} we 
show that the second component of 
the error satisfies the superconvergence bound stated in 
Theorem~\ref{thm:superconvergence}.
\begin{figure}
        \centering
        \begin{subfigure}[b]{0.48\textwidth}
	  \includegraphics[width=\textwidth]{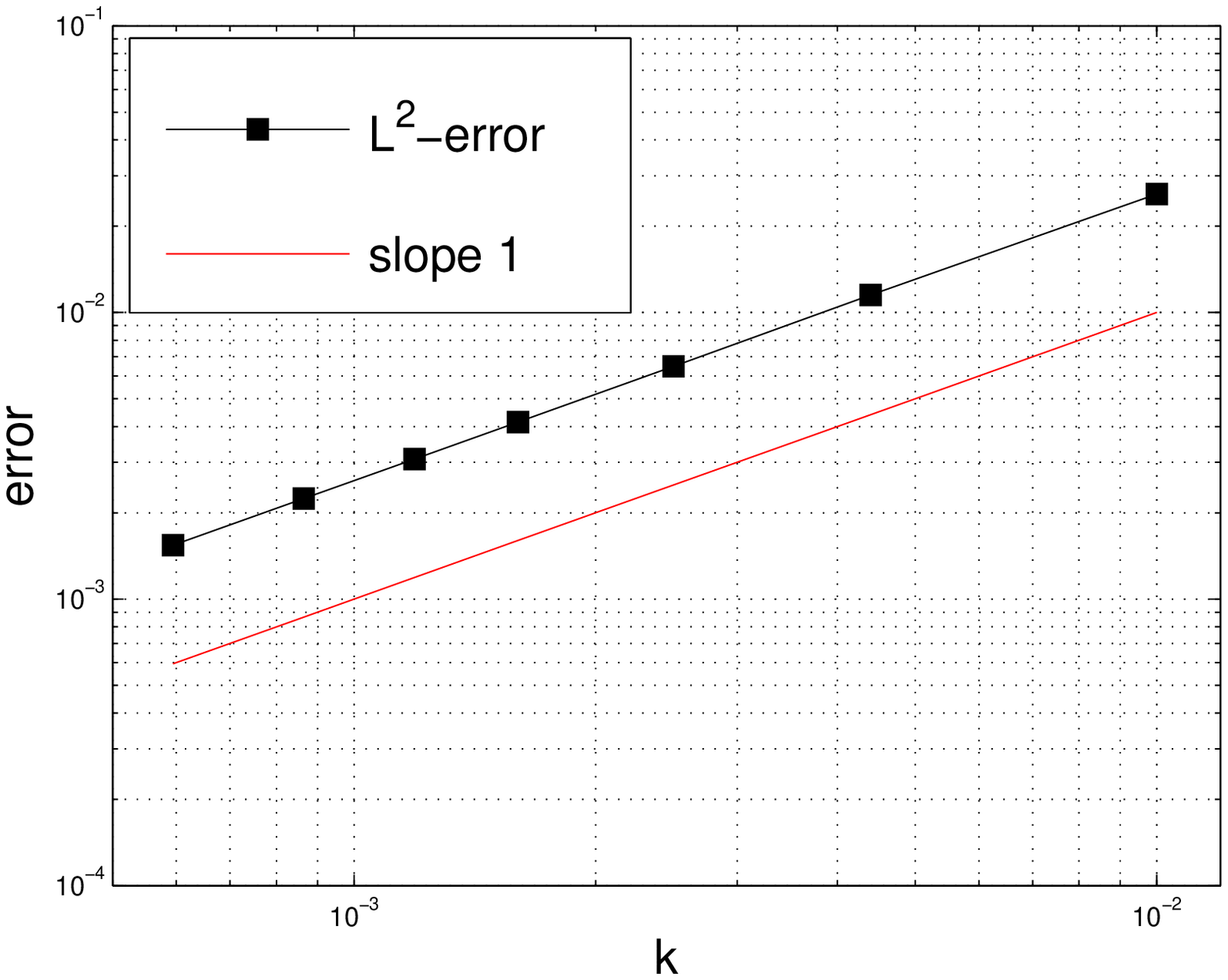}
	  \caption{Decay of the error.}
	  \label{fig:error_decay}
        \end{subfigure}
        \quad
        \begin{subfigure}[b]{0.48\textwidth}
                
\includegraphics[width=\textwidth]{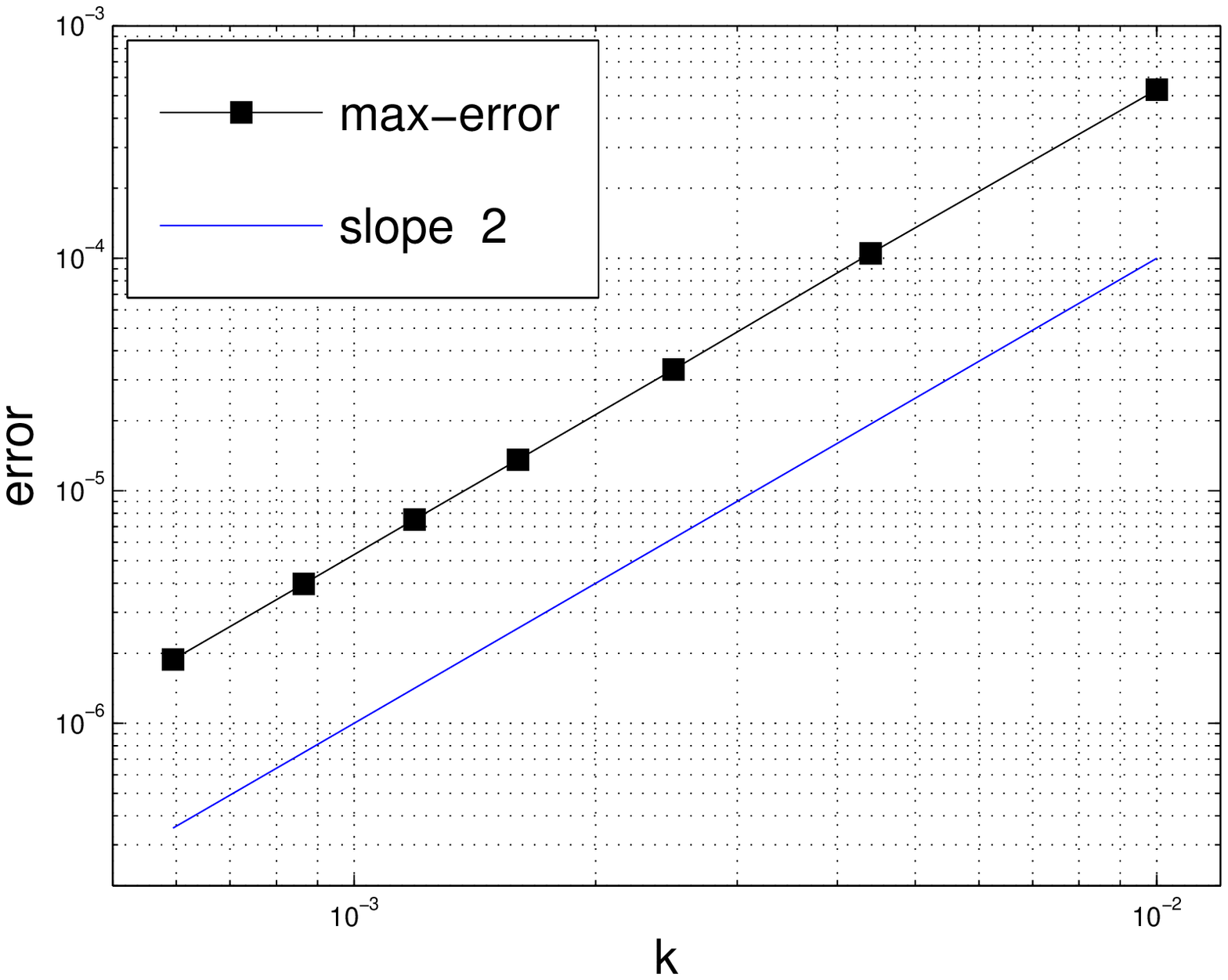}
                \caption{Superconvergence.}
                \label{fig:superconvergence}
        \end{subfigure}
        \caption{Numerical tests for Problem \eqref{eq:numerical_problem}.}
\end{figure}

\subsection{Two-dimensional test}
We test our scheme for the following problem on the space-time domain 
$(0,1)^2\times(0,1]$:
\begin{equation}\label{eq:numerical_problem_2d}
\begin{aligned}
&\dot{u}(\xi,\eta,t) - \Delta u(\xi,\eta,t) = 
\pi\,\sin(\pi\,\xi)\sin(\pi\eta)\Big( 
\cos(\pi\,t) + 2 \pi\,\sin(\pi\,t) \Big),\\
&u(0,\eta,t) = u(1,\eta,t) = 0,\quad  t\in[0,1],\,\eta \in 
(0,1), \\
&u(\xi,0,t) = u(\xi,1,t) = 0,\quad  t\in[0,1],\,\xi \in 
(0,1), \\
&u(\xi,\eta,0) = 0,\quad  (\xi,\eta) \in [0,1]^2,
\end{aligned}
\end{equation}
which has the solution $u(\xi,\eta,t) = 
\sin(\pi\,\xi)\sin(\pi\,\eta)\sin(\pi\,t)$.

In Figures~\ref{fig:error_decay_2d} and \ref{fig:superconvergence_2d} we report 
the analogous results to the ones presented in the one-dimensional case.

\begin{figure}
        \centering
        \begin{subfigure}[b]{0.48\textwidth}
	  \includegraphics[width=\textwidth]{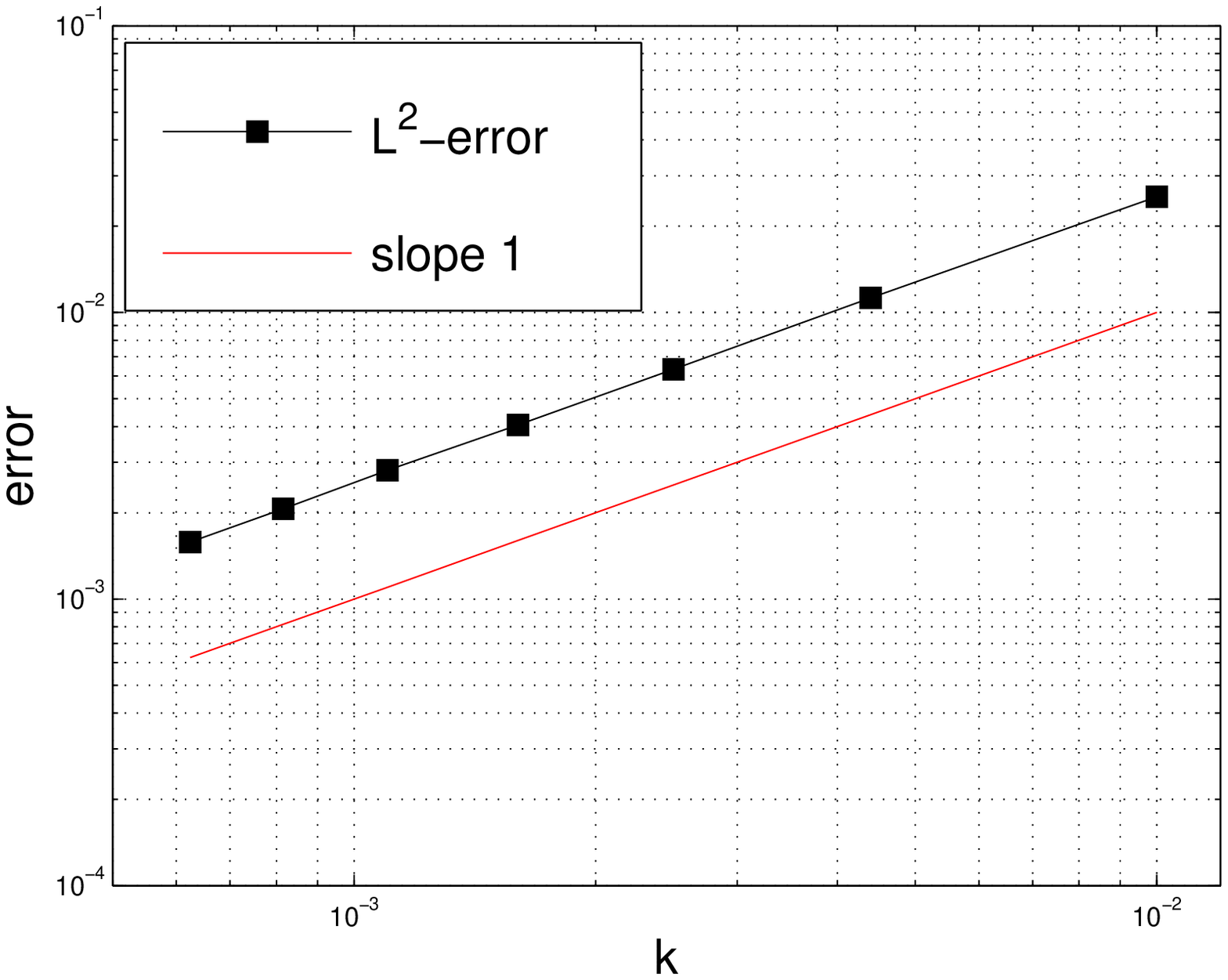}
	  \caption{Decay of the error.}
	  \label{fig:error_decay_2d}
        \end{subfigure}
        \quad
        \begin{subfigure}[b]{0.48\textwidth}
                
\includegraphics[width=\textwidth]{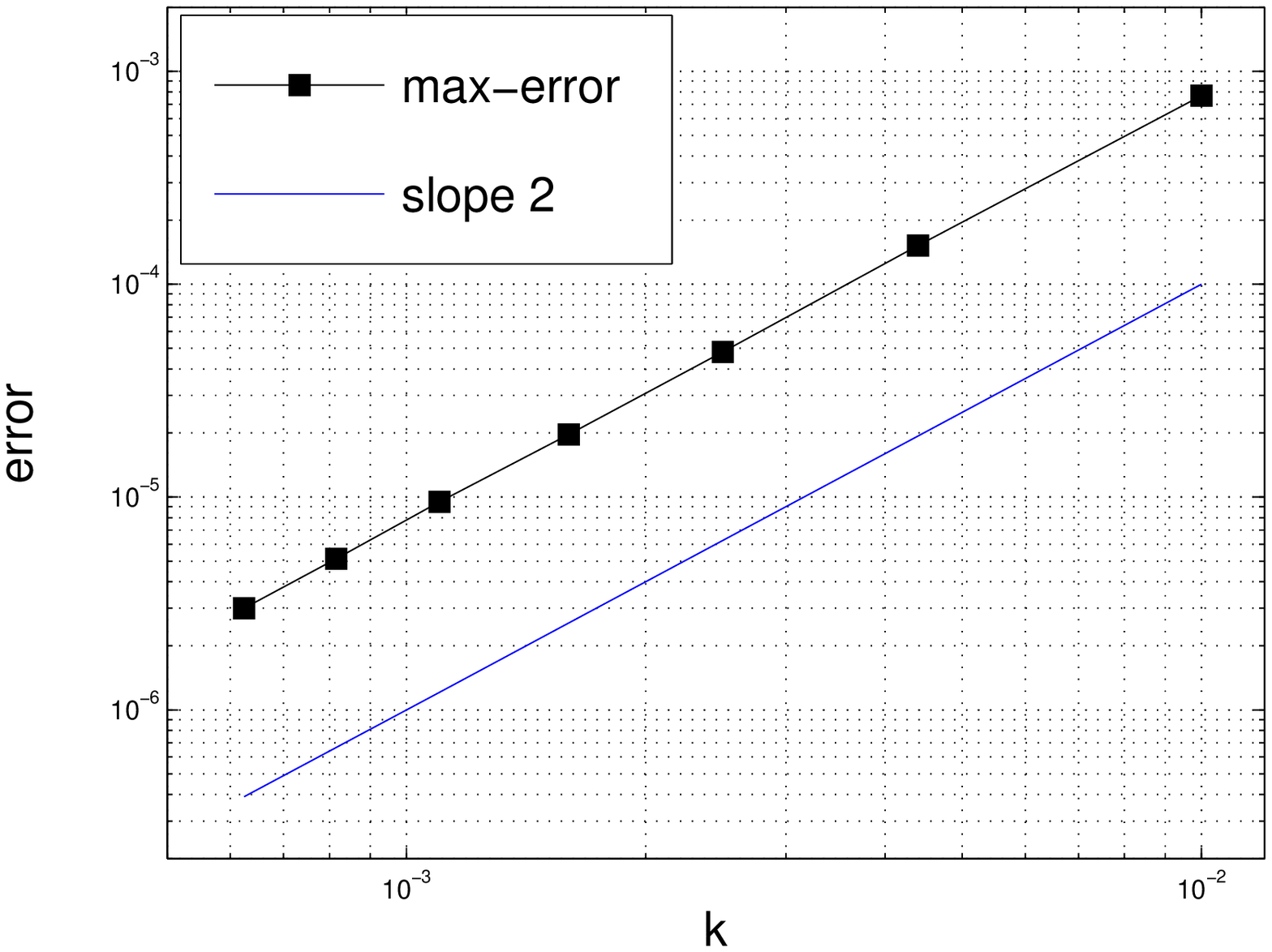}
                \caption{Superconvergence.}
                \label{fig:superconvergence_2d}
        \end{subfigure}
        \caption{Numerical tests for Problem \eqref{eq:numerical_problem_2d}.}
\end{figure}

\subsection{One-dimensional test, $q=1$}
In Figures~\ref{fig:error_decay_q1} and \ref{fig:superconvergence_q1} we can 
see 
the results of convergence and superconvergence when this scheme is used to 
solve Problem \eqref{eq:numerical_problem}. The convergence rate is optimal 
and consistent with our predictions.
\begin{figure}
        \centering
        \begin{subfigure}[b]{0.48\textwidth}
	  \includegraphics[width=\textwidth]{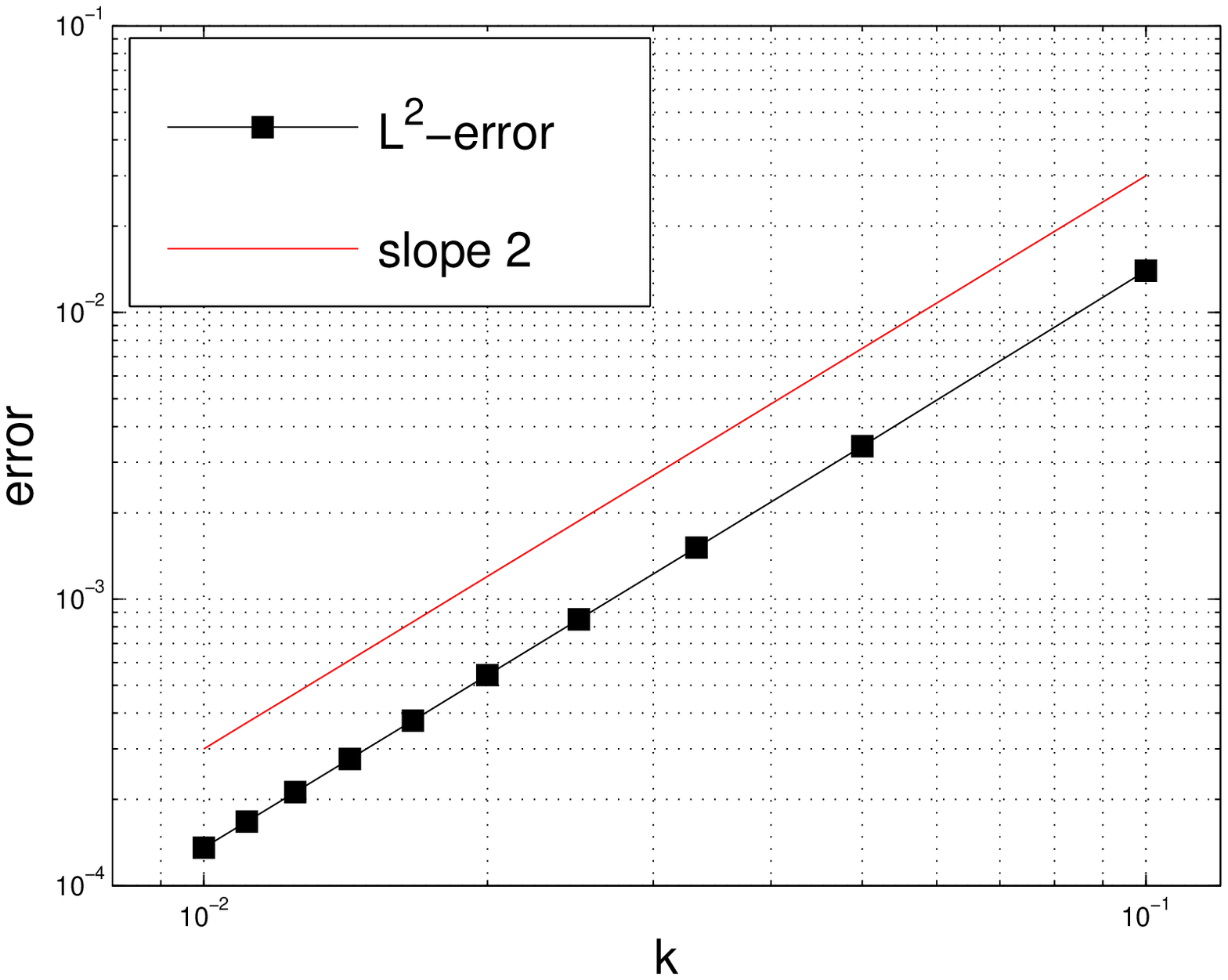}
	  \caption{Decay of the error.}
	  \label{fig:error_decay_q1}
        \end{subfigure}
        \quad
        \begin{subfigure}[b]{0.48\textwidth}
                
\includegraphics[width=\textwidth]{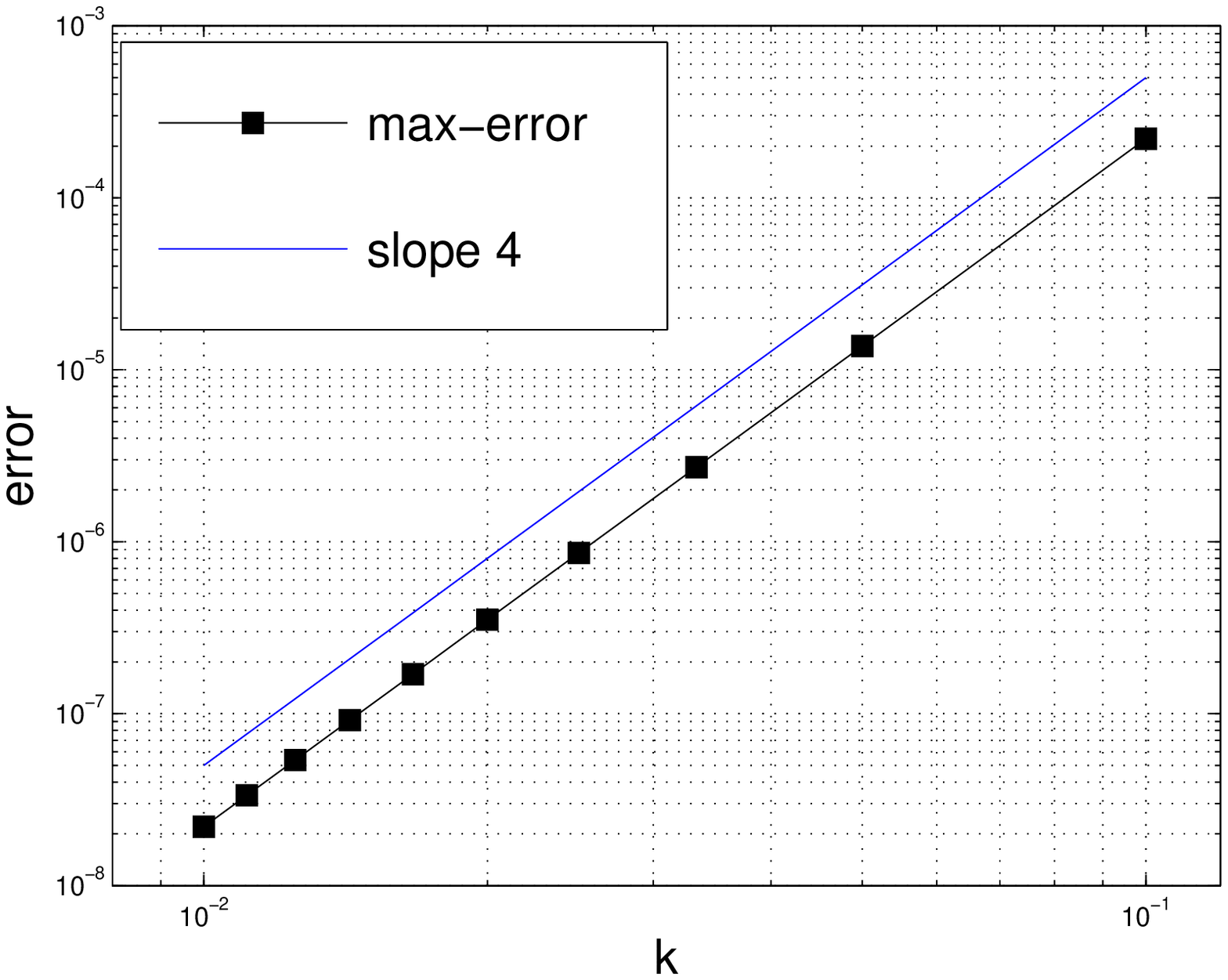}
                \caption{Superconvergence.}
                \label{fig:superconvergence_q1}
        \end{subfigure}
        \caption{Numerical tests for Problem \eqref{eq:numerical_problem}, 
$q=1$.}
\end{figure}

\subsection{One-dimensional test, low-regularity}

We investigate the behaviour of the error when the solution is not as smooth as 
we need to have superconvergence. We pick a problem such that $u$ has the 
first time-derivative which is square integrable, but not the second one.
More in detail, we choose $u$ equal to
$\abs{t-0.5}^{\frac{3 - \varepsilon }{2}} \sin(\pi \xi) $, where $\varepsilon$ 
is taken equal to $0.1$ in the case here investigated. 

In Figures~\ref{fig:error_decay_ns} and \ref{fig:superconvergence_ns} we can 
see 
the results of convergence and superconvergence when this scheme is used to 
solve our problem. The convergence rate for the first component of the error is 
optimal 
and consistent with our predictions. In this case the second component of the 
error does not superconverge and its rate of convergence behaves as the rate 
of convergence of the first component.

\begin{figure}
        \centering
        \begin{subfigure}[b]{0.48\textwidth}
	  \includegraphics[width=\textwidth]{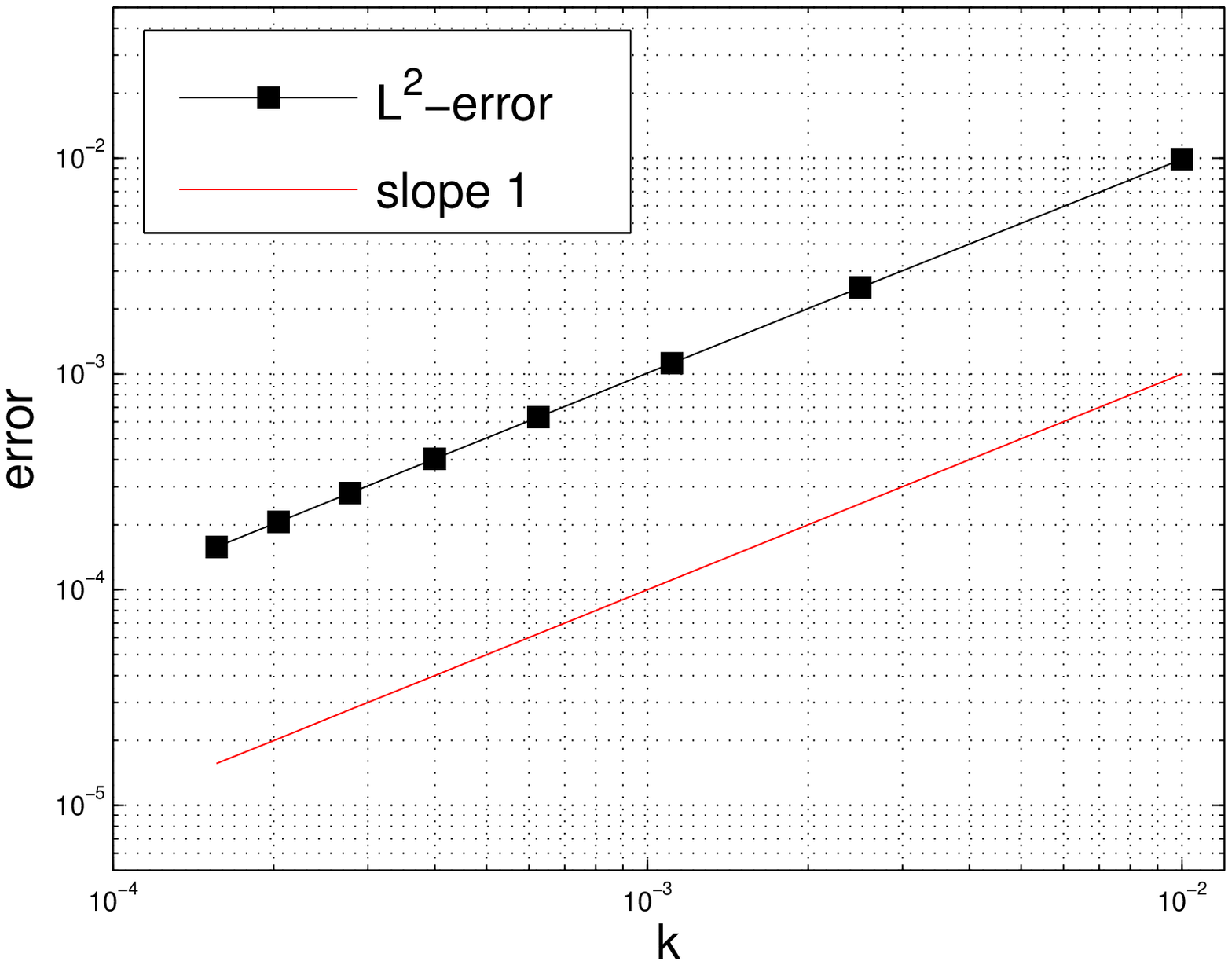}
	  \caption{Decay of the error.}
	  \label{fig:error_decay_ns}
        \end{subfigure}
        \quad
        \begin{subfigure}[b]{0.48\textwidth}
                
\includegraphics[width=\textwidth]{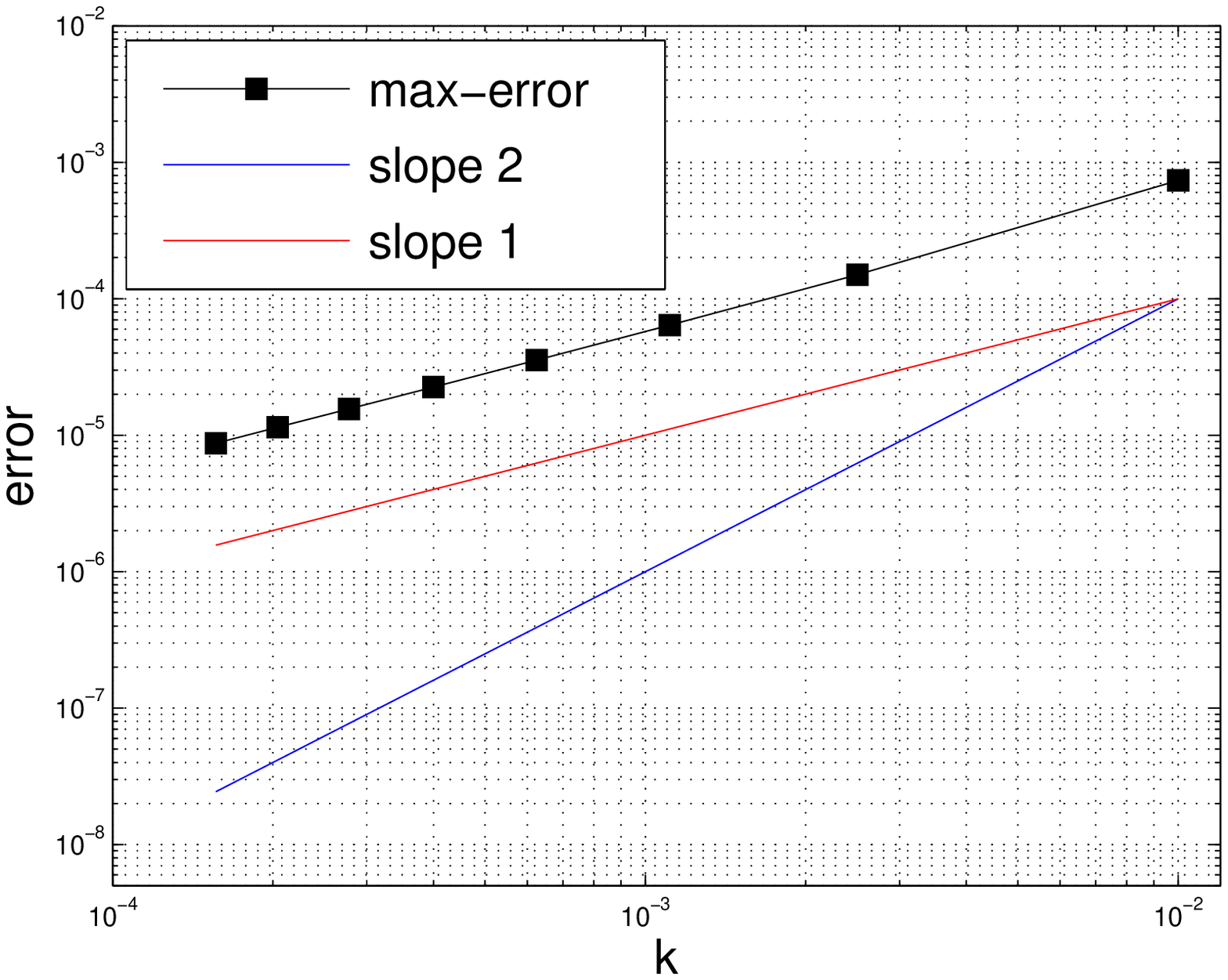}
                \caption{Non-superconvergence.}
                \label{fig:superconvergence_ns}
        \end{subfigure}
        \caption{Numerical tests for a problem with low-regular right-hand 
side.}
\end{figure}

%%%%%%%%%%%%%%%%%%%%%%%%%%%%%%%%%%%%%%%%%%%%%%%%%%%%%%%%%%%%%%%%%%%%%%%%%%
%%%%%%%%%%%%%%%%%%%%%%%%%%%%%%%%%%%%%%%%%%%%%%%%%%%%%%%%%%%%%%%%%%%%%%%%%%
%%%%%%%%%%%%%%%%%%%%%%%%%%%%%%%%%%%%%%%%%%%%%%%%%%%%%%%%%%%%%%%%%%%%%%%%%%
%%%%%%%%%%%%%%%%%%%%%%%%%%%%%%%%%%%%%%%%%%%%%%%%%%%%%%%%%%%%%%%%%%%%%%%%%%
%%									%%
%%		 F I N A L    R E M A R K S 				%%
%%									%%
%%%%%%%%%%%%%%%%%%%%%%%%%%%%%%%%%%%%%%%%%%%%%%%%%%%%%%%%%%%%%%%%%%%%%%%%%%
%%%%%%%%%%%%%%%%%%%%%%%%%%%%%%%%%%%%%%%%%%%%%%%%%%%%%%%%%%%%%%%%%%%%%%%%%%
%%%%%%%%%%%%%%%%%%%%%%%%%%%%%%%%%%%%%%%%%%%%%%%%%%%%%%%%%%%%%%%%%%%%%%%%%%
%%%%%%%%%%%%%%%%%%%%%%%%%%%%%%%%%%%%%%%%%%%%%%%%%%%%%%%%%%%%%%%%%%%%%%%%%%

\section{Final remarks}
In this article we have constructed a numerical scheme that produces a 
numerical solution under minimal regularity assumptions. The error of the 
solution has 
first been bounded in terms of the best possible approximation using the 
quasi-optimality theory, which does not require any further assumptions of 
regularity on the solution. The quasi-optimality constant that 
we obtain depends on the chosen discretization and requires the fulfilment of a 
certain CFL condition in order to have stability, consistently with the results 
in \cite{AndreevThesis} and \cite{Fra}.
We have shown that our scheme 
is of first order in time if we assume extra 
regularity, which means that the scheme is optimal with 
respect to the norm used to measure the error. Moreover, we have 
superconvergence at the points constituting the temporal grid, 
which means that the scheme is of second order in space and time. This 
further confirms the optimality of our method and its consistency with the 
known properties of the Crank--Nicolson scheme.
% change it and make it readable.
Since we do not need extra regularity to prove existence and 
uniqueness of a discrete solution, our scheme is in particular usable in 
contexts in which a smooth solution does not exist in the first place, and
this can, for example, constitute a novel approach for numerics to stochastic 
PDEs.

\bibliographystyle{alpha}
\bibliography{biblionew2}

\end{document}